\pdfoutput=1 
\documentclass[12pt,twoside]{article}
\usepackage{amsmath,amssymb,mathtools}
\usepackage{verbatim}
\usepackage{dsfont,bm,mathrsfs}
\usepackage{color}
\usepackage{graphicx}
\usepackage{amsthm}
\usepackage[shortlabels]{enumitem}
\usepackage{esint} 
\usepackage[title]{appendix}
\usepackage[T1]{fontenc}

\usepackage{xcolor}
\definecolor{darkgreen}{rgb}{0,0.75,0}
\definecolor{darkred}{rgb}{0.75,0,0}
\definecolor{darkmagenta}{rgb}{0.5,0,0.5}
\usepackage[colorlinks,citecolor=darkgreen,linkcolor=darkred,urlcolor=darkmagenta]{hyperref}
\newtheorem{thm}{Theorem}[section]

\newtheorem{cor}[thm]{Corollary}
\newtheorem{lem}[thm]{Lemma}
\newtheorem{prop}[thm]{Proposition}

\theoremstyle{definition}
\newtheorem{defn}[thm]{Definition}
\newtheorem{assum}[thm]{Assumption}
\newtheorem{rmk}[thm]{Remark}

\newtheorem{framework}[thm]{Framework}

\newtheorem*{notation}{Notation}
\numberwithin{equation}{section}
\numberwithin{figure}{section}

\allowdisplaybreaks

\newcommand{\norm}[1]{\left\lVert#1\right\rVert}
\newcommand{\trinorm}[1]{{\left\vert\kern-0.25ex\left\vert\kern-0.25ex\left\vert #1
    \right\vert\kern-0.25ex\right\vert\kern-0.25ex\right\vert}}
\newcommand{\indicator}[1]{\mathds{1}_{#1}} 
\newcommand{\closure}[1]{\overline{#1}}
\DeclareMathOperator*{\esssup}{ess\,sup}
\DeclareMathOperator*{\essinf}{ess\,inf}

\newcommand{\abs}[1]{\left\lvert#1\right\rvert} 

\newcommand{\supp}{\operatorname{supp}}

\newcommand{\diam}{\operatorname{diam}}
\newcommand{\dist}{\operatorname{dist}}

\newcommand{\Ext}{\mathsf{Ext}} 
\newcommand{\contfunc}{C}

\newcommand{\mr}[1]{{\tt \href{http://mathscinet.ams.org/mathscinet-getitem?mr=#1}{MR#1}}}
\newcommand{\arxiv}[1]{{\tt \href{http://arxiv.org/abs/#1}{arXiv:#1}}}
\newcommand{\doi}[1]{{\tt \href{https://doi.org/#1}{DOI:#1}}}

\begin{document}

	\font\titlefont=cmbx14 scaled\magstep1
	\title{\titlefont Characterizations of Sobolev functions via Besov-type energy functionals in fractals}

	\author{
	Ryosuke Shimizu
	}
	\date{May 5, 2025}
	\maketitle
	\vspace{-0.5cm}
	\begin{abstract}
		In the spirit of the ground-breaking result of Bourgain--Brezis--Mironescu, we establish some characterizations of Sobolev functions in metric measure spaces including fractals like the Vicsek set, the Sierpi\'{n}ski gasket and the Sierpi\'{n}ski carpet.
        As corollaries of our characterizations, we present equivalent norms on the Korevaar--Schoen--Sobolev space, and show that the domain of a $p$-energy form is identified with a Besov-type function space under a suitable $(p,p)$-Poincar\'e inequality, capacity upper bound and the volume doubling property.


		\
	\end{abstract}


\section{Introduction}
For a smooth bounded domain $\Omega \subset \mathbb{R}^{n}$, $p \in (1,\infty)$ and $u \in W^{1,p}(\Omega)$, in the celebrated paper \cite{BBM01}, Bourgain, Brezis and Mironescu established the following limit formula describing the Sobolev norm of $u$ as a suitable limit of the norms of the fractional Sobolev space $W^{\theta,p}(\Omega)$ as $\theta \uparrow 1$:
\begin{equation*}\label{e:intro.BBM}
    \lim_{\theta \uparrow 1}(1 - \theta)\int_{\Omega}\int_{\Omega}\frac{\abs{u(x) - u(y)}^{p}}{\abs{x - y}^{n + p\theta}}\,dx\,dy
    = K_{p,n}\int_{\Omega}\abs{\nabla u(x)}^{p}\,dx,
\end{equation*}
where $K_{p,n}$ is a constant depending only on $p,n$.
This formula provides us a new characterization (BBM type characterization) of $W^{1,p}(\Omega)$ in terms of the norms of $W^{\theta,p}(\Omega)$, $0 < \theta < 1$, and indicates new equivalent norms on $(1,p)$-Sobolev spaces on a metric measure space $(X,d,m)$ in place of $\mathbb{R}^{n}$ (\cite[Remark 6]{Bre02}). 
(We assume that $(X,d)$ is a locally compact, separable, complete, uniformly perfect metric space and that $m$ is a Radon measure on $X$ throughout this paper.)  
Such generalizations were done by \cite{DS19,Gor22,Han24,HP21+,LPZ24,Mun15} under the assumptions that $(X,d,m)$ is \emph{volume doubling} (Definition \ref{defn:doubling}) and supports a suitable \emph{Poincar\'{e} inequality} (see also \eqref{e:PIug}), and BBM type characterizations of the $(1,p)$-Sobolev spaces introduced by \cite{Che99,Haj96,HK00,KS93,KM98,Sha00}\footnote{These definitions give the same space when $(X,d)$ is complete, $m$ is volume doubling and $(X,d,m)$ supports a $(1,p)$-Poincar\'{e} inequality; see, e.g., \cite[Theorems 10.5.3 and 12.3.9]{HKST}.} were proved.
In particular, in \cite[Theorem 1.3]{LPZ24}, it was shown that there exists $C \ge 1$ such that for any $u \in L^{p}(X,m)$ and any \emph{strong $p$-extension domain} $\Omega \subset X$ (\cite[Definition 2.7]{LPZ24}),
\begin{align}\label{e:intro.BBM.PIspace}
    &C^{-1}E_{u,p}(\Omega)
    \le \liminf_{\theta \uparrow 1}(1 - \theta)\int_{\Omega}\int_{\Omega}\frac{\abs{u(x) - u(y)}^{p}}{d(x,y)^{p\theta}m(B(y,d(x,y)))}\,m(dx)\,m(dy) \nonumber \\
    &\quad \le \limsup_{\theta \uparrow 1}(1 - \theta)\int_{\Omega}\int_{\Omega}\frac{\abs{u(x) - u(y)}^{p}}{d(x,y)^{p\theta}m(B(y,d(x,y)))}\,m(dx)\,m(dy)
    \le CE_{u,p}(\Omega);
\end{align}
here $E_{u,p}(\Omega)$ denotes the $p$-th power of the $p$-energy norm of $u$ on $\Omega$ defined via the notion of upper gradient, i.e., $E_{u,p}(\Omega) \coloneqq \int_{\Omega}g_{u}^{p}\,dm$, where $g_{u}$ is a minimal $p$-weak upper gradient of $u$; see \cite[Section 6.3]{HKST} for the definition of $g_{u}$.
There are several results related to \eqref{e:intro.BBM.PIspace} that are not mentioned here; see, e.g., \cite{DS19,LPZ24,LPZ24b} and the references therein for further background.

The main aim of this paper is extending \eqref{e:intro.BBM.PIspace} to a setting including \emph{fractals} in view of recent progress on $(1,p)$-Sobolev spaces on fractals as done in \cite{CGQ22,HPS04,Kig23,MS25,Shi24}.
In this direction, under the volume doubling property and the so-called \emph{weak monotonicity condition} (see \eqref{e:intro.WM} below), Gao, Yu and Zhang \cite[Theorem 1.5]{GYZ23+} proved the following BBM type characterization in the case of $\Omega = X$:
\begin{align}\label{e:intro.BBM.GYZ}
    &C^{-1}\mathcal{E}_{p}(u)
    \le \liminf_{\alpha \uparrow \alpha_{p}(X)}(\alpha_{p}(X) - \alpha)\int_{X}\int_{X}\frac{\abs{u(x) - u(y)}^{p}}{d(x,y)^{p\alpha}m(B(y,d(x,y)))}\,m(dx)\,m(dy) \nonumber \\
    &\quad \le \limsup_{\alpha \uparrow \alpha_{p}(X)}(\alpha_{p}(X) - \alpha)\int_{X}\int_{X}\frac{\abs{u(x) - u(y)}^{p}}{d(x,y)^{p\alpha}m(B(y,d(x,y)))}\,m(dx)\,m(dy)
    \le C\mathcal{E}_{p}(u),
\end{align}
where $\alpha_{p}(X)$ is the \emph{$L^{p}$-Besov critical exponent} of $(X,d,m)$ (see Definition \ref{defn:criticalexponents}) and
\[
\mathcal{E}_{p}(u) \coloneqq \sup_{\varepsilon > 0}\int_{X}\fint_{B(y,\varepsilon)}\frac{\abs{u(x) - u(y)}^{p}}{\Psi(\varepsilon)}\,m(dx)\,m(dy).
\]
Here $\Psi(\varepsilon) \coloneqq \varepsilon^{p\alpha_{p}(X)}$ is called a \emph{scale function}, $\fint_{A}f\,dm \coloneqq \frac{1}{m(A)}\int_{A}f\,dm$ for $f \in L^{1}(X,m)$ and a Borel set $A$ of $X$ with $m(A) \in (0,\infty)$, and we say that the weak monotonicity condition holds if and only if there exists $C \ge 1$ such that for any $u \in L^{p}(X,m)$ with $\mathcal{E}_{p}(u) < \infty$,
\begin{equation}\label{e:intro.WM}
    \mathcal{E}_{p}(u) \le C\liminf_{\varepsilon \downarrow 0}\int_{X}\fint_{B(y,\varepsilon)}\frac{\abs{u(x) - u(y)}^{p}}{\Psi(\varepsilon)}\,m(dx)\,m(dy);
\end{equation}
see also \cite[Definition 4.5]{Bau24}.
(The BBM type characterization \eqref{e:intro.BBM.GYZ} is also used to construct good energy forms on fractals; see \cite{GYZ23,GY19}.)
A notable difference from the setting in the previous paragraph is that there is a possibility of $\alpha_{p}(X) > 1$; see \cite{KS.gc,KS.KS,MS25,Shi24} for this strict inequality for Sierpi\'{n}ski carpets and Sierpi\'{n}ski gaskets.
Because of this difference, many typical techniques relying on Lipschitz functions do not work, which makes the analysis harder.
Our main result (Theorem \ref{thm:main1} below), which is proved in an unified way using (a suitable Poincar\'{e}-type inequality and) good cutoff functions instead of Lipschitz functions, provide not only a localized version of \eqref{e:intro.BBM.GYZ} corresponding to \eqref{e:intro.BBM.PIspace} but also the weak monotonicity condition for Korevaar--Schoen type $p$-energy norms with a general scale function $\Psi$. 
Since the weak monotonicity condition is implied in our setting, the assumption in Theorem \ref{thm:main1} is stronger than that in \cite[Theorem 1.5]{GYZ23+}. We need this stronger assumption to ensure the existence of sufficiently nice cutoff functions, which play essential roles in the paper, especially in showing an analogous estimate to \eqref{e:intro.BBM.PIspace} (see \eqref{e:main1} or \eqref{e:BBM} later) that generalizes \cite[Theorem 1.5]{GYZ23+}.

\subsection{Framework}
We first introduce our framework and some conditions that are needed to state our main results.
In the following definition, we introduce $p$-energy forms satisfying the Lipschitz contraction property. 
\begin{defn}\label{defn.pEFLip}
	Let $(X,\mathcal{B},m)$ be a measure space and let $\mathcal{F}$ be a linear subspace of the class of all measurable functions on $X$, 
	and let $\mathcal{E} \colon \mathcal{F} \to [0,\infty)$.
	We say that $(\mathcal{E},\mathcal{F})$ is a \emph{$p$-energy form} on $(X,m)$ if and only if $\mathcal{E}^{1/p}$ is a seminorm on $\mathcal{F}$.
	In addition, the pair $(\mathcal{E},\mathcal{F})$ is said to satisfy the \emph{Lipschitz contraction property} if and only if for any $u \in \mathcal{F}$ and any $1$-Lipschitz map $\varphi \in C(\mathbb{R})$, i.e., $\abs{\varphi(x) - \varphi(y)} \le \abs{x-y}$ for any $x,y \in \mathbb{R}$, we have 
	\begin{equation}\label{e:Lipc}
		\varphi(u) \in \mathcal{F} \quad \text{and} \quad \mathcal{E}(\varphi(u)) \le \mathcal{E}(u). 
	\end{equation}
\end{defn}

Now we present our framework of a $p$-energy form and $p$-energy measures in this paper. We fix the following setting throughout this paper.
\begin{framework}\label{frame:EPEM}
    $(X,d)$ is a locally compact, separable, complete metric space with $\#X \ge 2$, and $m$ is a Radon measure on $X$ with full topological support. 
    Let $p \in (1,\infty)$, let $(\mathcal{E}_{p},\mathcal{F}_{p})$ be a $p$-energy form on $(X,m)$ with $\mathcal{F}_{p} \subset L^{p}(X,m)$, and let $\{ \Gamma_{p}\langle u \rangle \}_{u \in \mathcal{F}_{p}}$ be a family of finite Borel measures on $X$.
    We assume that the triple $(\mathcal{E}_{p},\mathcal{F}_{p},\Gamma_{p}\langle \,\cdot\, \rangle)$ satisfies the following conditions.
    \begin{enumerate}[label=\textup{(\arabic*)},align=left,leftmargin=*,topsep=2pt,parsep=0pt,itemsep=2pt]
        \item\label{it:Lip} $(\mathcal{E}_{p},\mathcal{F}_{p})$ satisfies the Lipschitz contraction property, and $\mathcal{E}_{p}^{-1}(0) \subset \mathbb{R}\indicator{X}$.
        \item\label{it:banach} We equip $\mathcal{F}_{p}$ with the norm $\norm{\,\cdot\,}_{L^{p}(X)} + \mathcal{E}_{p}(\,\cdot\,)^{1/p}$. Then $\mathcal{F}_{p}$ is a reflexive Banach space.
        \item\label{it:regular} $\mathcal{F}_{p} \cap C_{c}(X)$ is dense in $\mathcal{F}_{p}$ and in $C_{c}(X)$ which is equipped with the uniform norm. Here $C_{c}(X)$ is the set of compactly supported continuous functions on $X$. 
        \item\label{it:EMtri} For any $A \in \mathcal{B}(X)$, $(\Gamma_{p}\langle \,\cdot\,\rangle(A)^{1/p},\mathcal{F}_{p})$ is a $p$-energy form on $(X,m)$.
        \item\label{it:EMtotal} $\mathcal{E}_{p}(u) = \Gamma_{p}\langle u \rangle(X)$ for any $u \in \mathcal{F}_{p}$.
        \item\label{it:EMslocal} Let $u,v \in \mathcal{F}_{p} \cap C_{c}(X)$ and let $U$ be an open set of $X$. If $(u - v)|_{U} \in \mathbb{R}\indicator{U}$, then $\Gamma_{p}\langle u \rangle|_{U} = \Gamma_{p}\langle v \rangle|_{U}$. (Here, for Borel measure $\mu$ on $X$, $\mu|_{U}$ denotes the restriction of $\mu$ to the Borel sets of $U$.)
    \end{enumerate}
\end{framework}
\begin{rmk}\label{rmk:canonical.EM}
    There are several examples of $(\mathcal{E}_{p},\mathcal{F}_{p},\Gamma_{p}\langle \,\cdot\, \rangle)$ satisfying the conditions in Framework \ref{frame:EPEM}.
    See, e.g., \cite{BC23,CGQ22,CGYZ24+,KS.KS,Kig23,KO+,MS25,Shi24} for such examples on self-similar sets including Vicsek type fractals, the Sierpi\'{n}ski gasket and the Sierpi\'{n}ski carpet. 
	There has been recent progress on the existence of $p$-energy measures in \cite{Sas25+}, which provides us another framework in terms of $p$-energy forms only. Indeed, if a $p$-energy form $(\mathcal{E}_{p},\mathcal{F}_{p})$ on $(X,m)$ with $\mathcal{F}_{p} \subset L^p(X,m)$ and $\mathcal{E}_{p}^{-1}(0) \subset \mathbb{R}\indicator{X}$ satisfies the conditions (F1)-(F5) in \cite[Assumption 1.3]{Sas25+}, then there exists a family of finite Borel measures $\{ \Gamma_{p}\langle u \rangle \}_{u \in \mathcal{F}_{p}}$ on $X$ such that \ref{it:Lip}-\ref{it:EMslocal} of Framework \ref{frame:EPEM} hold; see \cite[Proposition 3.13]{KS.gc} and \cite[Theorem 1.4 and Proposition 2.6]{Sas25+}. Note that the $p$-energy measures constructed in \cite{Sas25+} are associated with $(\mathcal{E}_{p},\mathcal{F}_{p})$ through the representation formula in \cite[(5.3)]{Sas25+}. 
\end{rmk}

Let us emphasize that we do not deal with the case of $p \in \{ 1,\infty \}$ in this paper. This is mainly because the reflexivity of $\mathcal{F}_{p}$ is crucially used in the proofs of main results (Theorem \ref{thm:main1} and Corollary \ref{cor:main1.3}). In addition, the case of $p \in \{ 1,\infty \}$ is not included in the constructions of $p$-energy forms via discrete approximations in \cite{CGQ22,CGYZ24+,HPS04,Kig23,KO+,MS25,Shi24}. (The only exceptional is the case of the Vicsek set \cite{BC23}, where the case $p \in \{ 1,\infty \}$ is also established with the help of the tree structure of the Vicsek set. See also \cite[Remark 5.7]{BC23} for a remark in the case of $p = 1$.) We refer to \cite{ABCRST3,ABCRST.fractional} for some studies of BV class functions as the $L^1$ Korevaar--Schoen-type function space on fractals. 

Next we recall the doubling properties.
\begin{defn}[Doubling properties]\label{defn:doubling}
    \begin{enumerate}[label=\textup{(\arabic*)},align=left,leftmargin=*,topsep=2pt,parsep=0pt,itemsep=2pt]
        \item $(X,d)$ is said to be a \emph{doubling metric space} if and only if there exists $N_{\mathrm{D}} \in \mathbb{N}$ such that for any ball $B(x,r)$, where $B(x,r) \coloneqq \{ y \in X \mid d(x,y) < r \}$, can be covered by $N_{\mathrm{D}}$ balls whose radii are $r/2$.
        \item $(X,d,m)$ is said to satisfy the \emph{volume doubling property} \ref{e:VD} if and only if there exists $c_{\mathrm{D}} \in [1,\infty)$ such that
        \begin{equation}\label{e:VD}
            m(B(x,2r)) \le c_{\mathrm{D}}\,m(B(x,r)) \quad \text{for any $(x,r) \in X \times (0,\infty)$.} \tag{\textup{VD}}
        \end{equation}
    \end{enumerate}
\end{defn}
It is well known that if $(X,d,m)$ satisfies \ref{e:VD}, then $(X,d)$ is a doubling metric space. 

Henceforth, we fix a continuous increasing bijection $\Psi \colon (0,\infty) \to (0,\infty)$ such that for all $0 < r \le R$,
\begin{equation}\label{e:scalefunction}
    C_{\Psi}^{-1}\left(\frac{R}{r}\right)^{\beta_{1}} \le \frac{\Psi(R)}{\Psi(r)} \le C_{\Psi}\left(\frac{R}{r}\right)^{\beta_{2}}.
\end{equation}
Here $C_{\Psi} \in [1,\infty)$ and $\beta_{1},\beta_{2} \in (0,\infty)$ are constants independent of $r$ and $R$.

We introduce two key conditions in the following definition.
\begin{defn}[$(p,p)$-Poincar\'{e} inequality and $p$-capacity upper estimate]\label{defn:PI.cap}
    \begin{enumerate}[label=\textup{(\arabic*)},align=left,leftmargin=*,topsep=2pt,parsep=0pt,itemsep=2pt]
        \item We say that the \emph{$(p,p)$-Poincar\'{e} inequality} \ref{e:PI} holds if and only if there exist $C_{\mathrm{P}} \in (0,\infty)$ and $A_{\mathrm{P}} \in [1,\infty)$ such that for any $x \in X$, any $r \in (0,\infty)$ and any $u \in \mathcal{F}_{p}$, 
        \begin{equation}\label{e:PI}
            \int_{B(x,r)}\abs{u - u_{B(x,r)}}^{p}\,dm \le C_{\mathrm{P}}\Psi(r)\int_{B(x,A_{\mathrm{P}}r)}\,d\Gamma_{p}\langle u \rangle,  \tag{\textup{PI$_p$($\Psi$)}}
        \end{equation}
        where $u_{B(x,r)} \coloneqq \frac{1}{m(B(x,r))}\int_{B(x,r)}u\,dm$. 
        For $\beta \in (0,\infty)$, we say that \hyperref[e:PI]{\textup{PI$_p$($\beta$)}} holds if \ref{e:PI} holds with $\Psi(r) = r^{\beta}$.
        \item For Borel sets $E_{0},E_{1}$ of $X$, we set
        \[
        \mathcal{F}_{p}(E_{1},E_{0}) \coloneqq \{ u \in \mathcal{F}_{p} \mid \text{$u = 1$ $m$-a.e.\ on $E_{1}$ and $u = 0$ $m$-a.e.\ on $E_{0}$} \}.
        \]
        We say that the \emph{$p$-capacity upper estimate} \ref{e:capu} holds if and only if there exist $C_{\mathrm{cap}},A_{\mathrm{cap},1},A_{\mathrm{cap},2} \in (1,\infty)$ such that for any $x \in X$ and $0 < R < \diam(X,d)/A_{\mathrm{cap},2}$,
        \begin{equation}\label{e:capu}
            \inf\{ \mathcal{E}_{p}(u) \mid u \in \mathcal{F}_{p}(B(x,R), X \setminus B(x,A_{\mathrm{cap},1}R)) \}
            \le C_{\mathrm{cap}}\frac{m(B(x,R))}{\Psi(R)}. \tag{\textup{cap$_{p}$($\Psi$)$_{\le}$}}
        \end{equation}
        For $\beta \in (0,\infty)$, we say that \hyperref[e:capu]{\textup{cap$_p$($\beta$)$_{\le}$}} holds if \ref{e:capu} holds with $\Psi(r) = r^{\beta}$.
    \end{enumerate}
\end{defn}
\begin{rmk} 
	\begin{enumerate}[label=\textup{(\arabic*)},align=left,leftmargin=*,topsep=2pt,parsep=0pt,itemsep=2pt]
    	\item In the field of analysis on metric spaces, $(X,d,m)$ is said to support a $(p,p)$-Poincar\'{e} inequality if and only if for any \emph{upper gradient} $g$ of $u$ (see, e.g., \cite[Sections 6.2-6.4]{HKST} for the definition and background of upper gradients),
    \begin{equation}\label{e:PIug}
        \int_{B(x,r)}\abs{u - u_{B(x,r)}}^{p}\,dm \le C_{\mathrm{P}}\,r^{p}\int_{B(x,A_{\mathrm{P}}r)}g^{p}\,dm,
    \end{equation}
    which can be viewed as a special case of \hyperref[e:PI]{\textup{PI$_{p}$($p$)}}.
    Indeed, \eqref{e:PIug} implies that \hyperref[e:PI]{\textup{PI$_{p}$($p$)}} holds for $(\mathcal{E}_{p},\mathcal{F}_{p},\Gamma_{p}\langle \,\cdot\, \rangle)$ given by
    \[
    \mathcal{E}_{p}(u) \coloneqq \int_{X}g_{u}^{p}\,dm, \quad \mathcal{F}_{p} \coloneqq \{ u \in L^{p}(X,m) \mid \mathcal{E}_{p}(u) < \infty \}, \quad \Gamma_{p}\langle u \rangle(dx) \coloneqq g_{u}(x)^{p}\,m(dx), 
    \]
    where $g_{u}$ is a \emph{minimal $p$-weak upper gradient} of $u$ (see \cite[Section 6.3]{HKST}). 
    (The above $\mathcal{F}_{p}$ is called the \emph{Newton--Sobolev space}.)
    We can verify \hyperref[e:capu]{\textup{cap$_{p}$($p$)$_{\le}$}} for $(\mathcal{E}_{p},\mathcal{F}_{p},\Gamma_{p}\langle \,\cdot\, \rangle)$ by using a Lipschitz partition of unity (see \cite[pp.~104--105]{HKST}). 
    	\item In some settings including fractals like the Vicsek set, the Sierpi\'nski gasket and the Sierpi\'nski carpet, the proof of \hyperref[e:PI]{\textup{PI$_{p}$($\beta$)}} and \hyperref[e:capu]{\textup{cap$_p$($\beta$)$_{\le}$}} for a suitable parameter $\beta$ (called the $p$-walk dimension) can be found in, e.g., \cite[Propositions 5.28, 6.9 and 6.14]{KS.KS}, \cite[Proposition 8.21]{KS.gc} and \cite[Theorem 1.2]{MS25}. (Precisely, Eq. (5.30) in \cite[Proposition 5.28]{KS.KS} is not the Poincar\'{e} inequality in our sense, but it can be upgraded to \hyperref[e:PI]{\textup{PI$_{p}$($\beta$)}} by using \cite[(4.7) and Theorem 4.5]{KS.KS} and the regularity condition in Framework \ref{frame:EPEM}-\ref{it:regular}.) 
    	
    	We can also consider the case where the scale function $\Psi$ is not a power function, e.g., $\Psi(r) = r^{\beta}\indicator{\{r < 1\}} + r^{\beta'}\indicator{\{r \ge 1\}}$ with $\beta \neq \beta'$. See \cite[Figures 1,2 and p.~516]{BBK06} for fractal-like spaces satisfying \hyperref[e:PI]{\textup{PI$_{2}$($\Psi$)}} and \hyperref[e:capu]{\textup{cap$_2$($\Psi$)$_{\le}$}} with such a scale function. It is intriguing to ask which scale function $\Psi$ is allowed for the validity of \ref{e:PI} and \ref{e:capu}. In the case of $p = 2$, under a certain assumption on the behavior of measures of metric balls, a complete characterization in terms of a relation between $\Psi$ and a doubling scaling function $\Phi$ governing the volume growth of metric balls is given in \cite[Theorems 2.5 and 2.6]{Mur24+}. This result is a generalization of the previous result for random walks by \cite{Bar04}, where the case of both $\Psi$ and $\Phi$ are power functions is considered. (More precisely, in \cite{Bar04,Mur24+}, a metric measure space equipped with a strongly local regular Dirichlet form satisfying the full sub-Gaussian heat kernel estimates with the scale function $\Psi$ is constructed.)
    	Very recently, a $p$-analogue of there results was shown in \cite[Proposition 2.1 and Theorem 2.3]{Yan25+}. In particular, for any scale functions $\Phi$, $\Psi$ satisfying 
    	\begin{equation}\label{e:scale.characterization}
    		\frac{1}{C}\left(\frac{R}{r}\right)^{p} \le \frac{\Psi(R)}{\Psi(r)} \le C\left(\frac{R}{r}\right)^{p-1}\frac{\Phi(R)}{\Phi(r)} \quad \text{for any $0 < r \le R < \infty$,}
    	\end{equation}  
    	where $C \in [1,\infty)$ is a constant independent of $r$ and $R$, \cite[Theorem 2.3]{Yan25+} provides us an unbounded metric measure space $(X,d,m)$, $p$-energy form $(\mathcal{E}_p,\mathcal{F}_{p})$ and $p$-energy measures $\Gamma_{p}\langle \,\cdot\, \rangle$ on $(X,d,m)$ that satisfy the chain condition, volume growth condition V($\Phi$) (\cite[p.~4]{Yan25+}), \ref{e:PI} and \ref{e:capu}.  
    \end{enumerate}
\end{rmk} 

The following definition is a variant of \cite[Definition 2.5]{Mur24}.
\begin{defn}\label{defn.Flocal-ss}
    Let $U$ be a non-empty open subset of $X$.
    \begin{enumerate}[label=\textup{(\arabic*)},align=left,leftmargin=*,topsep=2pt,parsep=0pt,itemsep=2pt]
    	\item We define a linear subspace $\mathcal{F}_{p,\mathrm{loc}}(U)$ of $L^{0}(U,m|_{U})$ by
    	\begin{equation}\label{e:defn.Floc}
    		\mathcal{F}_{p,\mathrm{loc}}(U) \coloneqq
        	\biggl\{ u \in L^{0}(U,m|_{U}) \biggm|
        	\begin{minipage}{220pt}
            	$u = u^{\#}$ $m$-a.e.\ on $A$ for some $u^{\#} \in \mathcal{F}_{p}$ for each relatively compact open subset $A$ of $U$
        	\end{minipage}
        	\biggr\}.
    	\end{equation}
    	Here $L^{0}(U,m|_{U})$ is the set of all $m$-equivalence classes of $\mathbb{R}$-valued Borel measurable functions defined on $U$. 
    	\item For each $u \in \mathcal{F}_{p,\mathrm{loc}}(U)$, we further define a measure $\Gamma_{p,U}\langle u \rangle$ on $U$ as follows.
        We first define $\Gamma_{p,U}\langle u \rangle(E) \coloneqq \Gamma_{p}\langle u^{\#} \rangle(E)$ for each relatively compact Borel subset $E$ of $U$, with $A \subset U$ and $u^{\#} \in \mathcal{F}_{p}$ as in \eqref{e:defn.Floc} chosen so that $E \subset A$; this definition of $\Gamma_{p,U}\langle u \rangle(E)$ is independent of a particular choice of such $A$ and $u^{\#}$ by Framework \ref{frame:EPEM}-\ref{it:EMslocal}.
		We then define $\Gamma_{p,U}\langle u \rangle(E) \coloneqq \lim_{n \to \infty}\Gamma_{p,U}\langle u \rangle(E \cap A_{n})$ for each $E \in\{ B \cap U \mid \text{$B$ is a Borel set of $X$} \}$, where $\{ A_{n} \}_{n \in \mathbb{N}}$ is a non-decreasing sequence of relatively compact open subsets of $U$ such that $\bigcup_{n \in \mathbb{N}}A_{n} = U$; it is clear that this definition of $\Gamma_{p,U}\langle u \rangle(E)$ is independent of a particular choice of $\{ A_{n} \}_{n \in \mathbb{N}}$, coincides with the previous one when $E$ is relatively compact in $U$, and gives a Radon measure on $U$. 
        \item We define
        \begin{equation}\label{e:defn.FU}
            \mathcal{F}_{p}(U) \coloneqq \biggl\{ u \in \mathcal{F}_{p,\mathrm{loc}}(U) \biggm| \int_{U}\abs{u}^{p}\,dm + \int_{U}\,d\Gamma_{p,U}\langle u \rangle < \infty \biggr\}.
        \end{equation}
    \end{enumerate}
\end{defn}

Now we introduce an extension property for domains, which is inspired by \cite[Definition 2.7]{LPZ24}.
\begin{defn}\label{defn:sextdomain}
    Let $\Omega$ be an open subset of $X$ with $\Omega \neq X$.
    We say that $\Omega$ satisfies the extension property \hyperref[defn:sextdomain]{\textup{(E)}} (with respect to $(X,d,m,\mathcal{E}_{p},\mathcal{F}_{p},\Gamma_{p}\langle \,\cdot\, \rangle)$) if and only if for any $f \in \mathcal{F}_{p}(\Omega)$ there exists $F \in \mathcal{F}_{p}$ such that $f = F$ $m$-a.e.\ on $\Omega$ and $\Gamma_{p}\langle F \rangle(\partial\Omega) = 0$.
\end{defn}

\subsection{Main results}

Our main results will be described in terms of mollifiers as in \cite{LPZ24}.
In the rest of this paper, we fix a sequence of nonnegative measurable functions $\{ \rho_{\varepsilon} \}_{\varepsilon > 0}$ on $X \times X$.
Let us introduce the following assumptions on $\{ \rho_{\varepsilon} \}_{\varepsilon > 0}$.
\begin{assum}\label{assum:kernel.1}
    There exists $C_{\rho} \in (0,\infty)$ such that the following conditions hold.
    \begin{enumerate}[label=\textup{(\arabic*)},align=left,leftmargin=*,topsep=2pt,parsep=0pt,itemsep=2pt]
        \item\label{it:assum.nonlocal.upper} There exists $\{ d_{j}(\varepsilon) \}_{j \in \mathbb{N}} \subset [0,\infty)$ such that
        \begin{equation}\label{e:truncate.upper.const}
            \sum_{j = 1}^{\infty}d_{j}(\varepsilon) \le C_{\rho},
        \end{equation}
        and, for any $x,y \in X$ with $d(x,y) \le 1$,
        \begin{equation}\label{e:truncate.upper}
            \rho_{\varepsilon}(x,y) \le \sum_{j = 1}^{\infty}d_{j}(\varepsilon)\frac{\indicator{B(y,2^{-j+1}) \setminus B(y,2^{-j})}(x)}{m(B(y,2^{-j+1}))}.
        \end{equation}
        \item\label{it:assum.nonlocal.compl} For all $\delta > 0$,
        \begin{equation}\label{e:kernel.compl}
            \lim_{\varepsilon \downarrow 0}\esssup_{y \in X}\int_{X \setminus B(y,\delta)}\frac{\rho_{\varepsilon}(x,y)}{\Psi(d(x,y))}\,m(dx) = 0.
        \end{equation}
        \item\label{it:assum.nonlocal.lower} For any $x,y \in X$ with $d(x,y) \le 1$, either of the following two estimates hold:
        \begin{equation}\label{e:kernel.lower.1}
            \rho_{\varepsilon}(x,y) \ge C_{\rho}^{-1}\frac{\Psi(d(x,y))}{\Psi(\varepsilon)}\frac{\indicator{B(y,\varepsilon)}(x)}{m(B(y,\varepsilon))}
        \end{equation}
        or
        \begin{equation}\label{e:kernel.lower.2}
            \rho_{\varepsilon}(x,y) \ge \Psi(d(x,y))\frac{\nu_{\varepsilon}(d(x,y),\infty)}{m(B(y,d(x,y)))},
        \end{equation}
        where $\nu_{\varepsilon}$ is a Radon measure on $[0,\infty)$ for which
        \begin{equation}\label{e:kernel.measure}
            \liminf_{\varepsilon \downarrow 0}\int_{0}^{\delta}\Psi(t)\,\nu_{\varepsilon}(dt) \ge C_{\rho}^{-1} \quad \text{for any $\delta > 0$}.
        \end{equation}
    \end{enumerate}
\end{assum}

\begin{assum}\label{assum:kernel.2}
    There exists $C_{\rho} \in (0,\infty)$ such that the following conditions hold.
    \begin{enumerate}[label=\textup{(\arabic*)},align=left,leftmargin=*,topsep=2pt,parsep=0pt,itemsep=2pt]
        \item\label{it:assum.local.upper} There exists $\{ d_{j}(\varepsilon) \}_{j \in \mathbb{N}} \subset [0,\infty)$ satisfying \eqref{e:truncate.upper.const} and, for any $x,y \in X$,
        \begin{equation}\label{e:truncate.upper.2}
            \rho_{\varepsilon}(x,y) \le \sum_{j = 1}^{\infty}d_{j}(\varepsilon)\frac{\indicator{B(y,2^{-j+1}\varepsilon) \setminus B(y,2^{-j}\varepsilon)}(x)}{m(B(y,2^{-j+1}\varepsilon))},
        \end{equation}
        \item\label{it:assum.local.lower} Assumption \ref{assum:kernel.1}-\ref{it:assum.nonlocal.lower} holds.
    \end{enumerate}
\end{assum}
\begin{rmk}\label{rmk:slocalkernel}
    If $\{ \rho_{\varepsilon} \}_{\varepsilon > 0}$ satisfies Assumption \ref{assum:kernel.2} and $d(x,y) \ge \varepsilon$, then $\rho_{\varepsilon}(x,y) = 0$.
\end{rmk}
 
Before stating the main result of the paper, we present three main examples of $\{ \rho_{\varepsilon} \}_{\varepsilon > 0}$ satisfying Assumption \ref{assum:kernel.1} or \ref{assum:kernel.2}. (The verifications will be done in Subsection \ref{subsec:proofs}.) 
In the case of $\Psi(r) = r^{p\theta_{p}}$ for some $\theta_{p} \in [1,\infty)$, the family $\bigl\{ \rho_{\varepsilon}^{\mathrm{BBM}} \bigr\}_{\varepsilon > 0}$ given by 
\begin{equation}\label{e:BBMkernel}
	\rho_{\varepsilon}^{\mathrm{BBM}}(x,y) \coloneqq (\theta_{p} - \theta(\varepsilon))\frac{d(x,y)^{p\theta_{p}}}{d(x,y)^{p\theta(\varepsilon)}m(B(y,d(x,y)))}, \quad x,y \in X,   
\end{equation}
satisfies Assumption \ref{assum:kernel.1}. 
Here $\{ \theta(\varepsilon) \}_{\varepsilon > 0} \subset (0,\infty)$ is a sequence satisfying $\theta(\varepsilon) \uparrow \theta_p$ as $\varepsilon \downarrow 0$.  
Next, for any scale function $\Psi$ satisfying \eqref{e:scalefunction}, if $(X,d)$ is uniformly perfect (see Proposition \ref{prop:PI-up} for the definition), then the following families $\bigl\{ \rho_{\varepsilon}^{\mathrm{KS}} \bigr\}_{\varepsilon > 0}$ and $\bigl\{ \widehat{\rho}_{\varepsilon}^{\,\mathrm{KS}} \bigr\}_{\varepsilon > 0}$ satisfy Assumption \ref{assum:kernel.2}: 
\begin{equation}\label{e:KSkernel}
    \rho_{\varepsilon}^{\mathrm{KS}}(x,y) \coloneqq \frac{\Psi(d(x,y))}{\Psi(\varepsilon)}\frac{\indicator{B(y,\varepsilon)}(x)}{m(B(y,\varepsilon))} \quad \text{and} \quad \widehat{\rho}_{\varepsilon}^{\,\mathrm{KS}}(x,y) \coloneqq \frac{\indicator{B(y,\varepsilon)}(x)}{m(B(y,\varepsilon))}, \quad x,y \in X. 
\end{equation}

Now we present the main theorem in this paper.
\begin{thm}\label{thm:main1}
    Assume that \ref{e:VD}, \ref{e:PI} and \ref{e:capu} hold.
    Let $\{ \rho_{\varepsilon} \}_{\varepsilon > 0}$ satisfy Assumption \ref{assum:kernel.1} or \ref{assum:kernel.2}.
    Then there exists $C \in [1,\infty)$ such that for any open subset $\Omega$ of $X$ satisfying \hyperref[defn:sextdomain]{\textup{(E)}} and any $u \in \mathcal{F}_{p}(\Omega)$,
    \begin{align}\label{e:main1}
        C^{-1}\Gamma_{p}\langle u \rangle(\Omega)
        &\le \liminf_{\varepsilon \downarrow 0}\int_{\Omega}\int_{\Omega}\frac{\abs{u(x) - u(y)}^{p}}{\Psi(d(x,y))}\rho_{\varepsilon}(x,y)\,m(dx)\,m(dy) \nonumber \\
        &\le \limsup_{\varepsilon \downarrow 0}\int_{\Omega}\int_{\Omega}\frac{\abs{u(x) - u(y)}^{p}}{\Psi(d(x,y))}\rho_{\varepsilon}(x,y)\,m(dx)\,m(dy)
        \le C\Gamma_{p}\langle u \rangle(\Omega).
    \end{align}
	Moreover, if $\{ \rho_{\varepsilon} \}_{\varepsilon > 0}$ satisfies Assumption \ref{assum:kernel.2}, then for any $u \in \mathcal{F}_{p}$,   
    \begin{align}\label{e:WM}
        &\sup_{\varepsilon > 0}\int_{X}\int_{X}\frac{\abs{u(x) - u(y)}^{p}}{\Psi(d(x,y))}\rho_{\varepsilon}(x,y)\,m(dx)\,m(dy) \nonumber \\
        &\quad\le C\liminf_{\varepsilon \downarrow 0}\int_{X}\int_{X}\frac{\abs{u(x) - u(y)}^{p}}{\Psi(d(x,y))}\rho_{\varepsilon}(x,y)\,m(dx)\,m(dy).
    \end{align}
\end{thm}

We obtain the following asymptotic behaviors of the Sobolev seminorms as corollaries of Theorem \ref{thm:main1} by choosing $\{ \rho_{\varepsilon} \}_{\varepsilon > 0}$ as given in \eqref{e:BBMkernel} or \eqref{e:KSkernel}. 
\begin{cor}[BBM type characterization]\label{cor:main1.1}
    Assume that \ref{e:VD}, \hyperref[e:PI]{\textup{PI$_p$($p\theta_{p}$)}} and \hyperref[e:capu]{\textup{cap$_p$($p\theta_{p}$)$_{\le}$}} hold for some $\theta_{p} \ge 1$.
    Then there exists $C \in [1,\infty)$ such that for any open subset $\Omega$ of $X$ satisfying \hyperref[defn:sextdomain]{\textup{(E)}} and any $u \in \mathcal{F}_{p}(\Omega)$,
    \begin{align}\label{e:BBM}
        &C^{-1}\Gamma_{p}\langle u \rangle(\Omega)
        \le \liminf_{\theta \uparrow \theta_{p}}(\theta_{p} - \theta)\int_{\Omega}\int_{\Omega}\frac{\abs{u(x) - u(y)}^{p}}{d(x,y)^{p\theta}m(B(y,d(x,y)))}\,m(dx)\,m(dy) \nonumber \\
        &\,\le \limsup_{\theta \uparrow \theta_{p}}(\theta_{p} - \theta)\int_{\Omega}\int_{\Omega}\frac{\abs{u(x) - u(y)}^{p}}{d(x,y)^{p\theta}m(B(y,d(x,y)))}\,m(dx)\,m(dy)
        \le C\Gamma_{p}\langle u \rangle(\Omega).
    \end{align}
\end{cor}

\begin{cor}[Korevaar--Schoen type characterization]\label{cor:main1.2}
    Assume that \ref{e:VD}, \ref{e:PI} and \ref{e:capu} hold.
    Then there exists $C \in [1,\infty)$ such that for any open subset $\Omega$ of $X$ satisfying \hyperref[defn:sextdomain]{\textup{(E)}} and any $u \in \mathcal{F}_{p}(\Omega)$,
    \begin{align}\label{e:KS.local}
        &C^{-1}\Gamma_{p}\langle u \rangle(\Omega)
        \le \liminf_{r \downarrow 0}\int_{\Omega}\frac{1}{m(B(y,r))}\int_{\Omega \cap B(y,r)}\frac{\abs{u(x) - u(y)}^{p}}{\Psi(r)}\,m(dx)\,m(dy) \nonumber \\
        &\quad\le \limsup_{r \downarrow 0}\int_{\Omega}\frac{1}{m(B(y,r))}\int_{\Omega \cap B(y,r)}\frac{\abs{u(x) - u(y)}^{p}}{\Psi(d(x,y))}\,m(dx)\,m(dy)
        \le C\Gamma_{p}\langle u \rangle(\Omega).
    \end{align}
    Moreover, if $u \in \mathcal{F}_{p}$, then
    \begin{align}\label{e:KS}
        C^{-1}\mathcal{E}_{p}(u)
        &\le \liminf_{r \downarrow 0}\int_{X}\fint_{B(y,r)}\frac{\abs{u(x) - u(y)}^{p}}{\Psi(r)}\,m(dx)\,m(dy) \nonumber \\
        &\le \sup_{r > 0}\int_{X}\fint_{B(y,r)}\frac{\abs{u(x) - u(y)}^{p}}{\Psi(d(x,y))}\,m(dx)\,m(dy)
        \le C\mathcal{E}_{p}(u).
    \end{align}
\end{cor}

Similar to \cite[Theorem 7.1]{MS25}, we can identify $\mathcal{F}_{p}$ with Besov-type spaces as follows.
\begin{cor}[Determination of $\mathcal{F}_{p}$]\label{cor:main1.3}
    Assume that \ref{e:VD}, \ref{e:PI} and \ref{e:capu} hold.
    Let $\{ \rho_{\varepsilon} \}_{\varepsilon > 0}$ satisfy Assumption \ref{assum:kernel.1} or \ref{assum:kernel.2}.
    Then
    \begin{align}\label{e:determine}
        \mathcal{F}_{p} 
        &= \biggl\{ u \in L^{p}(X,m) \biggm| \liminf_{\varepsilon \downarrow 0}\int_{X}\int_{X}\frac{\abs{u(x) - u(y)}^{p}}{\Psi(d(x,y))}\rho_{\varepsilon}(x,y)\,m(dx)\,m(dy) < \infty \biggr\} \nonumber \\
        &= \biggl\{ u \in L^{p}(X,m) \biggm| \limsup_{\varepsilon \downarrow 0}\int_{X}\int_{X}\frac{\abs{u(x) - u(y)}^{p}}{\Psi(d(x,y))}\rho_{\varepsilon}(x,y)\,m(dx)\,m(dy) < \infty \biggr\}.
    \end{align}
\end{cor}
\begin{rmk}
    In the case $p = 2$, the determination of $\mathcal{F}_{2}$ as a Besov-type space was done by \cite{GHL03,Jon96,Kum00,KS05,PP99} under nice two-sided heat kernel estimates.
    It is conjectured that such nice heat kernel estimates are equivalent to the conjunction of \ref{e:VD}, \hyperref[e:PI]{\textup{PI$_{2}$($\Psi$)}} and \hyperref[e:PI]{\textup{cap$_{2}$($\Psi$)$_{\le}$}}. 
    This is called the \emph{resistance conjecture} and is still open except for some ``low-dimensional'' cases; see, e.g., \cite[Section 6.3]{Mur24} for further details.
\end{rmk}

The last main result ensures that there are many examples of domains satisfying \hyperref[defn:sextdomain]{\textup{(E)}}; see Section \ref{sec.ud} for the definition of uniform domains and for a proof of this theorem.
\begin{thm}\label{thm:ud.sext}
    Assume that \ref{e:VD}, \ref{e:PI} and \ref{e:capu} hold.
    Then any uniform domain $U$ of $X$ satisfies \hyperref[defn:sextdomain]{\textup{(E)}}. 
\end{thm}

The rest of this paper is organized as follows.
In Section \ref{sec:preli}, we recall some basic results used in this paper.
We prove the main results in Section \ref{sec:BBM}.
In Subsection \ref{subsec:upper}, we see that \ref{e:VD} and \ref{e:PI} imply upper bounds in Theorem \ref{thm:main1}.
In Subsection \ref{subsec:lower}, we show the existence of a controlled partition of unity under \ref{e:VD} and \ref{e:capu}, introduce discrete approximations of function in $L^{p}(X,m)$ using controlled partition of unity, and prove the lower bounds in Theorem \ref{thm:main1} using these tools and \ref{e:PI}.
The $L^{p}$-Besov critical exponents introduced in \cite{ABCRST1,ABCRST.fractional,Bau24,KSS24+} are discussed in Section \ref{sec:critical}.
In Section \ref{sec.ud}, we introduce uniform domains, Whitney coverings and the extension map. After stating basic properties of them, we prove Poincar\'{e} inequalities on uniform domains, a scale-invariant boundedness of the extension map and Theorem \ref{thm:ud.sext}.

\begin{notation}
    Throughout this paper, we use the following notation and conventions.
    \begin{enumerate}[label=\textup{(\arabic*)},align=left,leftmargin=*,topsep=2pt,parsep=0pt,itemsep=2pt]
        \item In general, $c_{i}$, $i \in \mathbb{N}$, will denote a positive constant depending on some unimportant parameters. By writing $C = C(\alpha,\beta,\dots)$ we indicate that $C$ is a constant depending only on $\alpha,\beta,\dots$.
	    \item For a set $A$, we let $\#A \in \mathbb{N} \cup \{ 0,\infty \}$ denote the cardinality of $A$. 
	    \item We set $\sup\emptyset \coloneqq 0$ and $\inf\emptyset \coloneqq \infty$. We write $a \vee b \coloneqq \max\{ a, b \}$, $a \wedge b \coloneqq \min\{ a, b \}$ and $a^{+} \coloneqq a \vee 0$ for $a, b \in [-\infty,\infty]$, and we use the same notation also for $[-\infty,\infty]$-valued functions and equivalence classes of them. 
		\item Let $K$ be a non-empty set. We define $\indicator{A}=\indicator{A}^{K} \in \mathbb{R}^{K}$ for $A \subset K$ by $\indicator{A}(x)\coloneqq \begin{cases} 1 \quad &\text{if $x \in A$,} \\ 0 \quad &\text{if $x \not\in A$,} \end{cases}$. 
		\item Let $(X,d)$ be a metric space and let $m$ be a Borel measure on $X$. The Borel $\sigma$-algebra of $X$ is denoted by $\mathcal{B}(X)$ and the closure of $A \subset X$ in $X$ by $\closure{A}^{X}$. We set $\supp_{X}[u] \coloneqq \closure{X \setminus u^{-1}(0)}^{X}$ for $u \in \contfunc(X)$, $B(x,r) \coloneqq \{ y \in X \mid d(x,y) < r \}$ for $(x,r) \in X \times (0,\infty)$, and $\diam(A,d) \coloneqq \sup_{x,y \in A}d(x,y)$ and $\dist(A,B) \coloneqq \inf\{ d(x,y) \mid x \in A, y \in B \}$ for subsets $A,B$ of $X$. When $A = \{ x \}$, we write $d(x,B) \coloneqq \dist(\{ x \}, B)$ for simplicity. Also, we set $f_{A} \coloneqq \fint_{A}f\,dm \coloneqq \frac{1}{m(A)}\int_{A}f\,dm$ for $f \in L^{1}(X,m)$ and $A \in \mathcal{B}(X)$ with $m(A) \in (0,\infty)$, and set $m|_{A} \coloneqq m|_{\mathcal{B}(X)|_{A}}$ for $A \in \mathcal{B}(X)$, where $\mathcal{B}(X)|_{A} \coloneqq \{ B \cap A \mid B \in \mathcal{B}(X) \}$.
    \end{enumerate}
\end{notation}

\section{Preliminary results}\label{sec:preli}
Recall the setting described in Framework \ref{frame:EPEM}. 
In this section, we collect some consequences of Framework \ref{frame:EPEM}. 
We first record an elementary proposition concerning \ref{e:VD}; see, e.g., \cite[Proposition 3.1]{BB}.
\begin{prop}\label{prop:top}
    Assume that \ref{e:VD} holds. 
    Then $\closure{A}^{X}$ is compact whenever $A$ is a bounded subset of $X$. 
    In particular, $m$ is $\sigma$-finite. 
\end{prop}

Let us recall the notion of a net in metric space.
\begin{defn}
    A maximal $\delta$-separated subset $A \subset X$ is called a \emph{$\delta$-net};, i.e., $A$ satisﬁes the following properties:
    \begin{enumerate}[label=\textup{(\alph*)},align=left,leftmargin=*,topsep=2pt,parsep=0pt,itemsep=2pt]
        \item $A$ is \emph{$\delta$-separated}, i.e., $d(x,y) \ge \delta$ whenever $x,y \in A$ with $x \neq y$.
        \item If $A' \subset X$ is also $\delta$-separated and $A \subset A'$, then $A' = A$.
    \end{enumerate}
\end{defn}

The following proposition collects properties implied by the Lipschitz contraction property of $(\mathcal{E}_{p},\mathcal{F}_{p})$. 
\begin{prop}\label{prop:Lipc.list}
	\begin{enumerate}[label=\textup{(\alph*)},align=left,leftmargin=*,topsep=2pt,parsep=0pt,itemsep=2pt]
		\item\label{it:Lip-bdd} Let $\Phi \colon \mathbb{R} \to \mathbb{R}$ satisfy $\Phi(0) = 0$. For a non-empty subset $A \subset \mathbb{R}$, define $\mathrm{Lip}(\Phi; A) \coloneqq \sup_{x,y \in A; x \neq y}\frac{\abs{\Phi(x)-\Phi(y)}}{\abs{x-y}}$. If $u \in \mathcal{F}_{p} \cap L^{\infty}(X,m)$ and $\mathrm{Lip}\bigl(\Phi; \bigl[-\norm{u}_{L^\infty(X,m)},\norm{u}_{L^\infty(X,m)}\bigr]\bigr) < \infty$, then 
        \begin{equation}\label{e:compos}
    		\Phi(u) \in \mathcal{F}_{p} \quad \text{and} \quad \mathcal{E}_{p}(\Phi(u)) \le \mathrm{Lip}\bigl(\Phi; \bigl[-\norm{u}_{L^\infty(X,m)},\norm{u}_{L^\infty(X,m)}\bigr]\bigr)^{p}\mathcal{E}_{p}(u).
    	\end{equation} 
    	In particular, $uv \in \mathcal{F}_{p} \cap L^{\infty}(X,m)$ whenever $u,v \in \mathcal{F}_{p} \cap L^{\infty}(X,m)$.
		\item\label{it:sadd.another} There exists $C = C(p) \ge 1$ such that for any $u,v \in \mathcal{F}_{p}$,
		\begin{equation}\label{e:sadd.another}
			\mathcal{E}_{p}(u \wedge v) + \mathcal{E}_{p}(u \vee v) \le C\bigl(\mathcal{E}_{p}(u) + \mathcal{E}_{p}(v)\bigr).
		\end{equation}
		\item\label{it:divide} Let $\delta, M > 0$ and let $u,v \in \mathcal{F}_{p}$ be non-negative functions such that $(u + v)|_{\{ u \neq 0 \}} \ge \delta$ and $u \le M$.
		Then there exists $C = C(p,\delta,M) \in (0,\infty)$ such that
		\begin{equation}\label{e:negative-power}
			\mathcal{E}_{p}\biggl(\frac{u}{u + v}\biggr) \le C\bigl(\mathcal{E}_{p}(u) + \mathcal{E}_{p}(v)\bigr).
		\end{equation}
	\end{enumerate}
\end{prop}
\begin{proof}
	We obtain \ref{it:Lip-bdd} by applying Framework \ref{frame:EPEM}-\ref{it:Lip} with a Lipschitz map $\widetilde{\Phi}$ satisfying $\widetilde{\Phi} = \Phi$ on $\bigl[-\norm{u}_{L^{\infty}(X,m)},\norm{u}_{L^{\infty}(X,m)}\bigr]$ and $\mathrm{Lip}(\widetilde{\Phi}; \mathbb{R}) = \mathrm{Lip}\bigl(\Phi; \bigl[-\norm{u}_{L^\infty(X,m)},\norm{u}_{L^\infty(X,m)}\bigr]\bigr)$. (Note that such a map $\widetilde{\Phi}$ exists by the McShane--Whitney extension lemma \cite[p.~99]{HKST}.) The latter assertion in \ref{it:Lip-bdd} is immediate by seeing that $uv = \frac{1}{4}\bigl[(u+v)^{2} - (u-v)^{2}\bigr]$. 
	One can show \ref{it:sadd.another} and \ref{it:divide} by following \cite[Proof of Proposition 6.26]{MS23+}. 
\end{proof}

The following lemma ensures that there are many good cutoff functions in $\mathcal{F}_{p} \cap C_{c}(X)$.
\begin{lem}\label{lem:capu.reduction}
    Assume that \ref{e:VD} and \ref{e:capu} hold.
    Then there exists $C = C(p,C_{\mathrm{cap}}) \in (0,\infty)$ such that for any $x \in X$ and any $0 < R < \diam(X)/A_{\mathrm{cap},2}$,
    \begin{equation}\label{e:conticapu}
        \inf\bigl\{ \mathcal{E}_{p}(v) \bigm| v \in \mathcal{F}_{p}(B(x,R), X \setminus B(x,A_{\mathrm{cap},1}R)) \cap C_{c}(X) \bigr\} \le C\,\frac{m(B(x,R))}{\Psi(R)}. 
    \end{equation}
\end{lem}
\begin{proof}
    Let $u \in \mathcal{F}_{p}(B(x,R),X \setminus B(x,A_{1}R))$ satisfy $\mathcal{E}_{p}(u) \le 2C_{\mathrm{cap}}\frac{m(B(x,R))}{\Psi(R)}$ and assume that $B(x,A_{\mathrm{cap},1}R) \neq X$.
    By Framework \ref{frame:EPEM}-\ref{it:regular}, there exists $v \in \mathcal{F}_{p} \cap C_{c}(X)$ such that $\mathcal{E}_{p}(v) \le 2\mathcal{E}_{p}(u)$ and 
    \[
    \norm{u - v}_{L^{p}}^{p} \le 4^{-p}\min\{ m(B(x,R)), m(B(x,2A_{\mathrm{cap},1}R) \setminus B(x,A_{\mathrm{cap},1}R)) \}.
    \]
    Set $\delta_{0} \coloneqq \essinf_{B(x,2A_{\mathrm{cap},1}R) \setminus B(x,A_{\mathrm{cap},1}R)}\abs{u - v}$ and $\delta_{1} \coloneqq \essinf_{B(x,R)}\abs{u - v}$; note that $\delta_{0} \vee \delta_{1} \le \frac{1}{4}$. 
    Now we define $\varphi \in C_{c}(X)$ by
    \[
    \varphi(y) \coloneqq
    \begin{cases}
        \left[2\left(v(y) - \frac{1}{4}\right)\right]^{+} \wedge 1 \quad &\text{for $y \in B(x,2A_{\mathrm{cap},1}R)$,}\\
        0 \quad &\text{for $y \in X \setminus B(x,2A_{\mathrm{cap},1}R)$.}\\
    \end{cases}
    \]
    It is clear that $\varphi = 1$ on $B(x,R)$ and $\varphi = 0$ on $X \setminus B(x,A_{\mathrm{cap},1}R)$.
    By Framework \ref{frame:EPEM}-\ref{it:Lip} \ref{it:EMtri} and \ref{it:EMslocal}, we have $\varphi \in \mathcal{F}_{p,\mathrm{loc}}(B(x,2A_{\mathrm{cap},1}R))$ and $\Gamma_{p}\langle \varphi \rangle(B(x,A_{\mathrm{cap},1}R)) \le 2^{p}\mathcal{E}_{p}(v)$. 
    Since any metric ball is relatively compact (Proposition \ref{prop:top}), we see from Framework \ref{frame:EPEM}-\ref{it:EMtotal}, \ref{it:EMslocal} that 
    \begin{align*}
        &\inf\bigl\{ \mathcal{E}_{p}(v) \bigm| v \in \mathcal{F}_{p}(B(x,R), X \setminus B(x,A_{\mathrm{cap},1}R)) \cap C_{c}(X) \bigr\} \\
        &= \inf\bigl\{ \Gamma_{p}\langle v \rangle(B(x,A_{\mathrm{cap},1}R)) \bigm| v \in \mathcal{F}_{p}(B(x,R), X \setminus B(x,A_{\mathrm{cap},1}R)) \cap C_{c}(X) \bigr\} \\
        &= \inf\biggl\{ \Gamma_{p}\langle v \rangle(B(x,A_{\mathrm{cap},1}R)) \biggm|
        \begin{minipage}{240pt}
            $v \in \mathcal{F}_{p,\mathrm{loc}}\bigl(B(x,2A_{\mathrm{cap},1}R)\bigr) \cap C_{c}(B(x,2A_{\mathrm{cap},1}R))$,
            $v = 1$ on $B(x,R)$, $v = 0$ on $X \setminus B(x,A_{\mathrm{cap},1}R)$
        \end{minipage}
        \biggr\} \\
        &\le \Gamma_{p}\langle \varphi \rangle(B(x,A_{\mathrm{cap},1}R)) 
        \le 2^{p + 2}C_{\mathrm{cap}}\frac{m(B(x,R))}{\Psi(R)}. 
    \end{align*}
    One can easily see that the left-hand side of \eqref{e:conticapu} is equal to $0$ when $B(x,A_{\mathrm{cap},1}R) = X$, so we complete the proof. 
\end{proof}

The following version of \ref{e:capu} can be shown by a standard covering argument. 
\begin{prop}\label{prop:capu.refine}
	Assume that \ref{e:VD} and \ref{e:capu} hold and let $\sigma \in (1,\infty)$. 
	Set $A_{\mathrm{cap},2}(\sigma) \coloneqq \frac{A_{\mathrm{cap},1} + 1}{(\sigma - 1)A_{\mathrm{cap},2}}$. 
	Then there exists $C_{\mathrm{cap}}(\sigma) = C(p,\sigma,C_{\Psi},\beta_{2},c_{\mathrm{D}},C_{\mathrm{cap}},A_{\mathrm{cap},1})$ such that for any $x \in X$ and any $0 < R < \diam(X)/A_{\mathrm{cap},2}(\sigma)$, 
	\begin{equation}\label{e:capu.refine}
		\inf\{ \mathcal{E}_{p}(u) \mid u \in \mathcal{F}_{p}(B(x,R), X \setminus B(x,\sigma R)) \}
            \le C_{\mathrm{cap}}(\sigma)\frac{m(B(x,R))}{\Psi(R)}.
	\end{equation}
\end{prop}
\begin{proof}
	Let $x \in X$ and $0 < R < \diam(X)/A_{\mathrm{cap},2}(\sigma)$. 
	Fix an $R'$-net $\mathcal{N}$ of $X$, where $R' \coloneqq \frac{\sigma - 1}{A_{\mathrm{cap},1} + 1}R$. 
	Note that if $y \in \mathcal{N}$ satisfies $B(x,R) \cap B(y,R') \neq \emptyset$, then $B(y,A_{\mathrm{cap},1}R') \subset B(x,\sigma R)$. 
	For each $y \in \mathcal{N}$, let $\psi_{y} \in \mathcal{F}_{p}(B(y,R'),X \setminus B(y,A_{\mathrm{cap},1}R'))$ be such that $\mathcal{E}_{p}(\psi_{y}) \le 2C_{\mathrm{cap}}\frac{m(B(y,R'))}{\Psi(R')}$. 
	By \ref{e:VD} and \eqref{e:scalefunction}, $\mathcal{E}_{p}(\psi_{y}) \le c_{1}\frac{m(B(x,R))}{\Psi(R)}$ for any $y \in \mathcal{N}$ satisfying$B(x,R) \cap B(y,R') \neq \emptyset$. 
	Hence the desired estimate \eqref{e:capu.refine} follows by considering 
	\[
	\psi \coloneqq \max_{y \in \mathcal{N}; B(x,R) \cap B(y,R') \neq \emptyset}\psi_{y}
	\]
	and using the doubling property of $(X,d)$ and \eqref{e:sadd.another}. 
\end{proof}

The following version of \ref{e:PI} is immediate from Framework \ref{frame:EPEM}-\ref{it:EMslocal}. 
\begin{prop}
    Assume that \ref{e:PI} holds and that all metric balls are relatively compact.
    Then for any $x \in X$, any $r \in (0,\infty)$ and any $u \in \mathcal{F}_{p,\mathrm{loc}}(B(x,A_{\mathrm{P}}r))$,
    \begin{equation}\label{e:PIloc}
        \int_{B(x,r)}\abs{u - u_{B(x,r)}}^{p}\,dm \le C_{\mathrm{P}}\Psi(r)\int_{B(x,A_{\mathrm{P}}\,r)}\,d\Gamma_{p,B(x,A_{\mathrm{P}}r)}\langle u \rangle. 
    \end{equation}
\end{prop}
%

Let us conclude this section by seeing a geometric implication of \ref{e:PI}. 
The following proposition is an analogue of \cite[Lemma 2.2 and Corollary 2.3]{Mur20}. 
See also Appendix \ref{sec:conn} for a stronger result, which is not needed to follow the main part of the paper. 
\begin{prop}\label{prop:PI-up}
	For $\varepsilon > 0$, define $d_{\varepsilon} \colon X \times X \to [0,\infty]$ by 
	\begin{equation}\label{e:defn.chainmet}
		d_{\varepsilon}(x,y) 
		\coloneqq \inf\Biggl\{ \sum_{i=0}^{N-1}d(x_{i},x_{i+1}) \Biggm| 
		\begin{minipage}{220pt}
			$N \in \mathbb{N}$, $\{ x_{i} \}_{i = 0}^{N} \subset X$, $x_0 = x$, $x_N = y$, and $d(x_{i},x_{i+1}) < \varepsilon$ for all $i = 0,\dots,N-1$
		\end{minipage}
		\Biggr\},  
	\end{equation}
	where the infimum taken over the empty set is defined to be infinite.   
	Assume that \ref{e:VD} and \ref{e:PI} hold. 
	Then 
	\begin{equation}\label{e:chainconn}
		d_{\varepsilon}(x,y) < \infty \quad \text{for all $x,y \in X$ and for all $\varepsilon > 0$.}
	\end{equation}
	In particular, $(X,d)$ is uniformly perfect, i.e., there exists $\sigma \in (0,1)$ such that $B(x,r) \setminus B(x,\sigma r) \neq \emptyset$ whenever $B(x,r) \neq X$. 
\end{prop}
\begin{proof}
	The argument in \cite[Corollary 2.3]{Mur20} shows that \eqref{e:chainconn} implies that $(X,d)$ is uniformly perfect, so we prove \eqref{e:chainconn}. 
	The proof of \eqref{e:chainconn} is a minor modification of \cite[Lemma 2.2]{Mur20}. 
	Let $x \in X$ and $\varepsilon > 0$. 
	As in \cite[Proof of Lemma 2.2]{Mur20}, define $U_{x} \coloneqq \{ y \in X \mid d_{\varepsilon}(x,y) < \infty \}$ and $V_{x} \coloneqq X \setminus U_{x}$. 
	Then both $U_{x}$ and $V_{x}$ are open subsets of $X$.  
	Let $N$ be a $\varepsilon/2$-net in $(X,d)$ and, for each $z \in N$, let $\varphi_{z} \in \mathcal{F}_{p} \cap \contfunc_{c}(X)$ be such that $0 \le \varphi_{z} \le 1$, $\varphi_{z}|_{B(z,\varepsilon/2)} = 1$ and $\supp_{X}[\varphi_{z}] \subset B(z,\varepsilon)$. 
	For any relatively compact open subset $A$, define $N_{A} \coloneqq \{ z \in N \cap U_x \mid A \cap B(z,\varepsilon) \neq \emptyset \}$.  
	We have from $\#N_{A} < \infty$ and Proposition \ref{prop:Lipc.list}-\ref{it:sadd.another} that $\varphi_{A}^{\#} \coloneqq \min_{z \in N_{A}}\varphi_{z}$ satisfies $\varphi_{A}^{\#} \in \mathcal{F}_{p} \cap \contfunc_{c}(X)$. 
	Moreover, since $\dist(U_{x}, V_{x}) \ge \varepsilon$ (see \cite[(2.1)]{Mur20}) and $\{ B(z, \varepsilon/2) \}_{z \in N_{A}}$ covers $U_{x} \cap A$, we have $\varphi_{A}^{\#} = \indicator{U_{x}}$ on $A$. 
	Now we suppose $V_{x} \neq \emptyset$. 
	Then we can pick $(z,r) \in X \times (0,\infty)$ so that $U_{x} \cap B(z,r) \neq \emptyset$ and $V_{x} \cap B(z,r) \neq \emptyset$. 
	Recall that $B(z,A_{\mathrm{P}}r)$, where $A_{\mathrm{P}}$ is the constant in \ref{e:PI}, is a relatively compact open subset of $X$ by Proposition \ref{prop:top}. 
	By \ref{e:PI} with $u = \varphi_{B(z,A_{\mathrm{P}}r)}^{\#}$, 
	\begin{align*}
		0 < \int_{B(z,r)}\abs{u - u_{B(z,r)}}^{p}\,dm 
		&\le C_{\mathrm{P}}\Psi(r)\int_{B(z,A_{\mathrm{P}}r)}\,d\Gamma_{p}\langle u \rangle \\
		&= C_{\mathrm{P}}\Psi(r)\bigl[\Gamma_{p}\langle u \rangle(B(z,A_{\mathrm{P}}r) \cap U_{x}) + \Gamma_{p}\langle u \rangle(B(z,A_{\mathrm{P}}r) \cap V_{x})\bigr]. 
	\end{align*}
	However, both $\Gamma_{p}\langle u \rangle(B(z,A_{\mathrm{P}}r) \cap U_{x})$ and $\Gamma_{p}\langle u \rangle(B(z,A_{\mathrm{P}}r) \cap V_{x})$ are equal to $0$ by $u = \indicator{U_{x}}$ on $B(z,A_{\mathrm{P}}r)$ and Framework \ref{frame:EPEM}-\ref{it:EMslocal}.
	This is a contradiction and hence $V_{x} = \emptyset$, which means that \eqref{e:chainconn} holds. 
\end{proof} 

\section{BBM type characterizations}\label{sec:BBM}
In this section, we prove the first main result Theorem \ref{thm:main1}. 
We start with a few preparations. 
For any open subset $U \subsetneq X$ and any $\delta > 0$, we set
\begin{equation}\label{e:nbd-innner}
    U(\delta) \coloneqq \bigcup_{x \in U}B(x,\delta) \quad \text{and} \quad U_{\delta} \coloneqq \{ x \in U \mid d(x,X \setminus U) > \delta \};
\end{equation}
note that $U_{\delta}(\delta) \subset U$.

Next, we recall a few fundamental lemmas, which will be frequently used in this section.
\begin{lem}[{\cite[(3.1)]{LPZ24}}]
    For any $u \in L^{p}(X,m)$ and any $(z,r) \in X \times (0,\infty)$,
    \begin{equation}\label{e:doublevar}
        \int_{B(z,r)}\int_{B(z,r)}\abs{u(x) - u(y)}^{p}\,m(dx)\,m(dy)
        \le 2^{p}m(B(z,r))\int_{B(z,r)}\abs{u - u_{B(z,r)}}^{p}\,dm.
    \end{equation}
\end{lem}

\begin{lem}[{\cite[Lemma 3.5]{LPZ24}}]
    Assume that \ref{e:VD} holds.
    For any open subset $U \subset X$ and any measurable function $h \colon X \times X \to [0,\infty]$ satisfying $h(x,y) = 0$ for $m \times m$-a.e.\  $(x,y) \in X \times X$ with $d(x,y) \ge \delta > 0$,
    \begin{align}\label{e:triint}
        &\int_{U}\int_{X}h(x,y)\,m(dx)\,m(dy) \nonumber\\
        &\quad \le c_{\mathrm{D}}\int_{U(2\delta)}\frac{1}{m(B(z,\delta))}\iint_{B(z,2\delta) \times B(z,2\delta)}h(x,y)\,m(dx)\,m(dy)\,m(dz).
    \end{align}
\end{lem}

\subsection{Upper bound of Theorem \ref{thm:main1}}\label{subsec:upper}
To get the upper estimate in \eqref{e:main1}, we first prove the following estimate, which is an analogue of \cite[Proposition 4.1]{LPZ24}. 
\begin{prop}\label{prop:preupper}
    Assume that \ref{e:VD}, \ref{e:PI} hold and that $\{ \rho_{\varepsilon} \}_{\varepsilon > 0}$ satisfies Assumption \ref{assum:kernel.1}-\ref{it:assum.nonlocal.upper}.
    Then there exists $C = C(p,C_{\rho},C_{\Psi},\beta_{2},c_{\mathrm{D}},C_{\mathrm{P}},A_{\mathrm{P}}) \in (0,\infty)$ such that for any open set $U$ of $X$, any $\varepsilon > 0$, any $R \in (0,1)$ and any $u \in \mathcal{F}_{p,\mathrm{loc}}(U(8A_{\mathrm{P}}R))$,
    \begin{equation}
        \int_{U}\int_{B(y,R)}\frac{\abs{u(x) - u(y)}^{p}}{\Psi(d(x,y))}\rho_{\varepsilon}(x,y)\,m(dx)\,m(dy)
        \le C\Gamma_{p}\langle u \rangle(U(8A_{\mathrm{P}}R)).
    \end{equation}
\end{prop}
\begin{proof}
    For simplicity, we set $B_{x,j} \coloneqq B(x,2^{-j})$ for any $(x,j) \in X \times \mathbb{Z}$.
    Let $R \in (0,1)$ and let $j(R) \in \mathbb{N}$ satisfy $2^{-j(R)} \le R < 2^{-j(R) + 1}$.
    For any $j \in \mathbb{N}$ with $j \ge j(R)$, we have the following estimate, which is essentially the same as \cite[(4.4)]{LPZ24}:
    \begin{align}\label{e:upper.1}
        &\int_{U}\int_{B(y,R)}\frac{\abs{u(x) - u(y)}^{p}}{\Psi(d(x,y))}\frac{\indicator{B_{y,j - 1} \setminus B_{y,j}}(x)}{m(B_{y,j - 1})}\,m(dx)\,m(dy) \nonumber \\
        &\overset{\eqref{e:triint}}{\le} c_{\mathrm{D}}\int_{U(2^{-j+2})}\frac{1}{m(B_{z,j - 1})}\iint_{B_{z,j - 2} \times B_{z,j - 2}}\frac{\abs{u(x) - u(y)}^{p}}{\Psi(d(x,y))}\frac{\indicator{X \setminus B_{y,j}}(x)}{m(B_{y,j-1})}\,m(dx)\,m(dy)\,m(dz).
    \end{align}
    Note that the double integral in the left-hand side in the above inequality is equal to $0$ when $j < j(R)$.
    By using \eqref{e:doublevar} and \ref{e:PI}, we further obtain the following estimate: 
    \begin{align}\label{e:upper.2}
        &\frac{1}{m(B_{z,j - 1})}\iint_{B_{z,j - 2} \times B_{z,j - 2}}\frac{\abs{u(x) - u(y)}^{p}}{\Psi(d(x,y))}\frac{\indicator{X \setminus B_{y,j}}(x)}{m(B_{y,j-1})}\,m(dx)\,m(dy) \nonumber \\
        &\le c_{1}\frac{\Psi(2^{-j+2})}{\Psi(2^{-j})}\frac{\Gamma_{p}\langle u \rangle(B(z,2^{-j+2}A_{\mathrm{P}}))}{m(B_{z,j-2})} 
        \le c_{2}\frac{\Gamma_{p}\langle u \rangle(B(z,2^{-j+2}A_{\mathrm{P}}))}{m(B_{z,j-2})}.
    \end{align}
    Similar to \cite[Proof of Proposition 4.1]{LPZ24}, by \eqref{e:upper.1} and  \eqref{e:upper.2}, we have 
    \begin{equation}\label{e:upper.3}
    	\int_{U}\int_{B(y,R)}\frac{\abs{u(x) - u(y)}^{p}}{\Psi(d(x,y))}\frac{\indicator{B_{y,j - 1} \setminus B_{y,j}}(x)}{m(B_{y,j - 1})}\,m(dx)\,m(dy) 
    	\le c_{3}\Gamma_{p}\langle u \rangle(U(2^{-j+3}A_{\mathrm{P}})).
    \end{equation}
    Now we have from \eqref{e:truncate.upper.const}, \eqref{e:truncate.upper} and \eqref{e:upper.3} that
    \begin{align}
        &\int_{U}\int_{B(y,R)}\frac{\abs{u(x) - u(y)}^{p}}{\Psi(d(x,y))}\rho_{\varepsilon}(x,y)\,m(dx)\,m(dy) \nonumber \\
        &\le \sum_{j = 1}^{\infty}d_{j}(\varepsilon)\int_{U}\int_{B(y,R)}\frac{\abs{u(x) - u(y)}^{p}}{\Psi(d(x,y))}\frac{\indicator{B_{y,j - 1} \setminus B_{y,j}}(x)}{m(B_{y,j - 1})}\,m(dx)\,m(dy) \nonumber \\
        &\le c_{3}\sum_{j = j(R)}^{\infty}d_{j}(\varepsilon)\Gamma_{p}\langle u \rangle(U(2^{-j+3}A_{\mathrm{P}})) 
        \le c_{3}C_{\rho}\Gamma_{p}\langle u \rangle(U(8A_{\mathrm{P}}R)),
    \end{align}
    completing the proof.
\end{proof}

The following theorem shows the desired upper estimate in Theorem \ref{thm:main1} under Assumption \ref{assum:kernel.1}-\ref{it:assum.nonlocal.upper},\ref{it:assum.nonlocal.compl}.
\begin{thm}\label{thm:upper.ext}
    Assume that \ref{e:VD}, \ref{e:PI} hold and that $\{ \rho_{\varepsilon} \}_{\varepsilon > 0}$ satisfies Assumption \ref{assum:kernel.1}-\ref{it:assum.nonlocal.upper},\ref{it:assum.nonlocal.compl}.
    Then there exists $C = C(p,C_{\rho},C_{\Psi},\beta_{2},c_{\mathrm{D}},C_{\mathrm{P}},A_{\mathrm{P}}) \in (0,\infty)$ such that for any open subset $\Omega$ of $X$ satisfying \hyperref[defn:sextdomain]{\textup{(E)}} and any $u \in \mathcal{F}_{p}(\Omega)$,
    \begin{equation}\label{e:main.upper.ext}
        \limsup_{\varepsilon \downarrow 0}\int_{\Omega}\int_{\Omega}\frac{\abs{u(x) - u(y)}^{p}}{\Psi(d(x,y))}\rho_{\varepsilon}(x,y)\,m(dx)\,m(dy)
        \le C\Gamma_{p}\langle u \rangle(\Omega).
    \end{equation}
    Moreover, if $u \in \mathcal{F}_{p}$, then \eqref{e:main.upper.ext} holds with $\Omega = X$ as well. 
\end{thm}
\begin{proof}
	Let $R \in (0,1)$ and set $D_{\varepsilon}(u; x,y) \coloneqq \frac{\abs{u(x) - u(y)}^{p}}{\Psi(d(x,y))}\rho_{\varepsilon}(x,y)$ for $x,y \in X$.  
    Let us divide the double integral of $D_{\varepsilon}(u; x,y)$ into the following three parts: 
    \begin{align}\label{e:main.upper.divide}
        &\int_{\Omega}\int_{\Omega}D_{\varepsilon}(u; x,y)\,m(dx)\,m(dy) \nonumber \\
        &= \left(\int_{\Omega}\int_{\Omega \setminus B(y,R)} +\int_{\Omega_{8A_{\mathrm{P}}R}}\int_{B(y,R)} + \int_{\Omega \setminus \Omega_{8A_{\mathrm{P}}R}}\int_{\Omega \cap B(y,R)}\right)D_{\varepsilon} (u; x,y)\,m(dx)\,m(dy). 
    \end{align}
    For the first term, the same argument as in \cite[Proof of Theorem 4.8]{LPZ24} shows 
    \begin{align}\label{e:main.upper.1st}
        &\int_{\Omega}\int_{\Omega \setminus B(y,R)}D_{\varepsilon}(u; x,y)\,m(dx)\,m(dy) \nonumber \\
        &\le 2^{p}\norm{u}_{L^{p}(X)}^{p}\esssup_{y \in \Omega}\int_{\Omega \setminus B(y,R)}\frac{\rho_{\varepsilon}(x,y)}{\Psi(d(x,y))}\,m(dx)
        \xrightarrow[\varepsilon \downarrow 0]{} 0,
    \end{align}
    where we used \eqref{e:kernel.compl} in the last line. 
    For the second term, by Proposition \ref{prop:preupper}, we get 
    \begin{equation}\label{e:main.upper.2nd}
        \int_{\Omega_{8A_{\mathrm{P}}R}}\int_{B(y,R)}D_{\varepsilon}(u; x,y)\,m(dx)\,m(dy)
        \le C\Gamma_{p}\langle u \rangle(\Omega_{8A_{\mathrm{P}}R}(8A_{\mathrm{P}}R)) 
        \le C\Gamma_{p}\langle u \rangle(\Omega).
    \end{equation}
    Lastly, we estimate the third term.
    Since $\Omega$ satisfies \hyperref[defn:sextdomain]{\textup{(E)}}, there exists $v \in \mathcal{F}_{p}$ such that $v|_{\Omega} = u$ and $\Gamma_{p}\langle v \rangle(\partial\Omega) = 0$.
    Set $U \coloneqq \Omega \setminus \Omega_{8A_{\mathrm{P}}R}$ and note that $U(8A_{\mathrm{P}}R) \subset \Omega(8A_{\mathrm{P}}R) \setminus \Omega_{16A_{\mathrm{P}}R}$.
    Then we see from Proposition \ref{prop:preupper} that
    \begin{equation*}
        \int_{\Omega \setminus \Omega_{8A_{\mathrm{P}}R}}\int_{\Omega \cap B(y,R)}D_{\varepsilon}(u; x,y)\,m(dx)\,m(dy) 
        \le C\Gamma_{p}\langle v \rangle(\Omega(8A_{\mathrm{P}}R) \setminus \Omega_{16A_{\mathrm{P}}R}).
    \end{equation*}
    Since $R \in (0,1)$ is arbitrary and $\Gamma_{p}\langle v \rangle(\partial\Omega) = 0$, we get 
    \begin{equation}\label{e:main.upper.3rd}
        \lim_{R \downarrow 0}\int_{\Omega \setminus \Omega_{8A_{\mathrm{P}}R}}\int_{\Omega \cap B(y,R)}D_{\varepsilon}(u; x,y)\,m(dx)\,m(dy) = 0.
    \end{equation}
    By \eqref{e:main.upper.divide}, \eqref{e:main.upper.1st}, \eqref{e:main.upper.2nd} and \eqref{e:main.upper.3rd}, we obtain \eqref{e:main.upper.ext}.

	The case $u \in \mathcal{F}_{p}$ and $\Omega = X$ can be shown with minor modifications by observing that 
    \begin{equation*}
        \int_{X}\int_{X}D_{\varepsilon}(u; x,y)\,m(dx)\,m(dy) 
        = \left(\int_{X}\int_{X \setminus B(y,R)} + \int_{X}\int_{B(y,R)}\right)D_{\varepsilon}(u; x,y)\,m(dx)\,m(dy). 
    \end{equation*}
    Indeed, \eqref{e:kernel.compl} implies that the first term goes to $0$ as $\varepsilon \downarrow 0$ similarly to \eqref{e:main.upper.1st}, and Proposition \ref{prop:preupper} yields that the second term is bounded above by $C\Gamma_{p}\langle u \rangle(X) = C\mathcal{E}_{p}(u)$. 
\end{proof}

Now we move to the case where $\{ \rho_{\varepsilon} \}_{\varepsilon > 0}$ satisfies Assumption \ref{assum:kernel.2}.
The following proposition is an analogue of Proposition \ref{prop:preupper}.
\begin{prop}\label{prop:preupper.KS}
    Assume that \ref{e:VD}, \ref{e:PI} hold and that $\{ \rho_{\varepsilon} \}_{\varepsilon > 0}$ satisfies Assumption \ref{assum:kernel.2}-\ref{it:assum.local.upper}.
    Then there exists $C = C(p,C_{\rho},C_{\Psi},\beta_{2},c_{\mathrm{D}},C_{\mathrm{P}},A_{\mathrm{P}}) \in (0,\infty)$ such that for any open set $U$ of $X$, any $\varepsilon > 0$ and any $u \in \mathcal{F}_{p,\mathrm{loc}}(U(2A_{\mathrm{P}}\varepsilon))$,
    \begin{equation}\label{e:upper.local}
        \int_{U}\int_{X}\frac{\abs{u(x) - u(y)}^{p}}{\Psi(d(x,y))}\rho_{\varepsilon}(x,y)\,m(dx)\,m(dy)
        \le C\Gamma_{p}\langle u \rangle(U(2A_{\mathrm{P}}\varepsilon)).
    \end{equation}
\end{prop}
\begin{proof}
    For each $j \in \mathbb{N}$, let $\{ z_{k} \}_{k \in I_{j}(U)} \subset U$ be a $(2^{-j+1}\varepsilon)$-net of $U$.
    Then we have 
    \begin{align}\label{e:pvar.net}
		&\int_{U}\int_{B(y,2^{-j+1}\varepsilon) \setminus B(y,2^{-j}\varepsilon)}\frac{\abs{u(x) - u(y)}^{p}}{\Psi(d(x,y))m(B(y,2^{-j+1}r))}\,m(dx)\,m(dx) \nonumber \\
		&\le \sum_{k \in I_{j}(U)}\frac{c_{\mathrm{D}}^{2}}{\Psi(2^{-j}\varepsilon)m(B(z_{k},2^{-j+1}\varepsilon))}\iint_{B(z_{k},2^{-j+2}\varepsilon) \times B(z_{k},2^{-j+2}\varepsilon)}\abs{u(x) - u(y)}^{p}\,m(dx)\,m(dy) \nonumber \\
        &\overset{\eqref{e:doublevar}}{\le} \sum_{k \in I_{j}(U)}\frac{2^{p}c_{\mathrm{D}}^{3}}{\Psi(2^{-j}\varepsilon)}\int_{B(z_{k},2^{-j+2}\varepsilon)}\abs{u(x) - u_{B(z_{k},2^{-j+2}\varepsilon)}}^{p}\,m(dx) \nonumber \\
		&\overset{\ref{e:PI}}{\le} 2^{p}c_{\mathrm{D}}^{3}C_{\mathrm{P}}\frac{\Psi(2^{-j+2}\varepsilon)}{\Psi(2^{-j}\varepsilon)}\sum_{k \in I_{j}(U)}\int_{B(z_{k},2^{-j+2}A_{\mathrm{P}}\varepsilon)}\,d\Gamma_{p}\langle u \rangle
        \le c_{1}\int_{U(2A_{\mathrm{P}}\varepsilon)}\,d\Gamma_{p}\langle u \rangle,
	\end{align}
    where we used the doubling property of $(X,d)$ in the last inequality.
   Since we have from \eqref{e:truncate.upper.2} that
    \begin{align*}
        &\int_{U}\int_{X}\frac{\abs{u(x) - u(y)}^{p}}{\Psi(d(x,y))}\rho_{\varepsilon}(x,y)\,m(dx)\,m(dy) \\
        &\le \sum_{j = 1}^{\infty}d_{j}(\varepsilon)\int_{U}\int_{B(y,2^{-j+1}\varepsilon) \setminus B(y,2^{-j}\varepsilon)}\frac{\abs{u(x) - u(y)}^{p}}{\Psi(d(x,y))m(B(y,2^{-j+1}\varepsilon))}\,m(dx)\,m(dx), 
    \end{align*}
    we obtain \eqref{e:upper.local} by using \eqref{e:pvar.net} and \eqref{e:truncate.upper.const}.
\end{proof}

Now we obtain the desired upper estimate in Theorem \ref{thm:main1} under Assumption \ref{assum:kernel.2}-\ref{it:assum.local.upper}.
\begin{thm}\label{thm:upper.local.KSext}
    Assume that \ref{e:VD}, \ref{e:PI} hold and that $\{ \rho_{\varepsilon} \}_{\varepsilon > 0}$ satisfies Assumption \ref{assum:kernel.2}-\ref{it:assum.local.upper}.
    Then there exists $C = C(p,C_{\rho},C_{\Psi},\beta_{2},c_{\mathrm{D}},C_{\mathrm{P}},A_{\mathrm{P}}) \in (0,\infty)$ such that for any open subset $\Omega$ of $X$ satisfying \hyperref[defn:sextdomain]{\textup{(E)}} and any $u \in \mathcal{F}_{p}(\Omega)$,
    \begin{equation}\label{e:upper.local.KSext}
        \limsup_{\varepsilon \downarrow 0}\int_{\Omega}\int_{\Omega}\frac{\abs{u(x) - u(y)}^{p}}{\Psi(d(x,y))}\rho_{\varepsilon}(x,y)\,m(dx)\,m(dy)
        \le C\Gamma_{p}\langle u \rangle(\Omega).
    \end{equation}
    Moreover, if $u \in \mathcal{F}_{p}$, then we have 
    \begin{equation}\label{e:sup.upper}
        \sup_{\varepsilon > 0}\int_{X}\int_{X}\frac{\abs{u(x) - u(y)}^{p}}{\Psi(d(x,y))}\rho_{\varepsilon}(x,y)\,m(dx)\,m(dy)
        \le C\mathcal{E}_{p}(u).
    \end{equation}
\end{thm}
\begin{proof}
    The same argument as in the proof of Theorem \ref{thm:upper.ext} proves \eqref{e:upper.local.KSext} since by \eqref{e:truncate.upper.2} (see also Remark \ref{rmk:slocalkernel}) we have
    \[
    \esssup_{y \in \Omega}\int_{\Omega \setminus B(y,R)}\frac{\rho_{\varepsilon}(x,y)}{\Psi(d(x,y))}\,m(dx)
    = 0
    \]
    whenever $0 < \varepsilon < R$. 
    If $u \in \mathcal{F}_{p}$, then we get \eqref{e:sup.upper} by applying \eqref{e:upper.local} with $U = X$.
\end{proof}

\subsection{Lower bounds of Theorem \ref{thm:main1}}\label{subsec:lower}
For a proof of the lower bound in Theorem \ref{thm:main1}, which is based on a similar argument as in \cite[Section 5]{LPZ24} using approximations by discrete convolutions, we need a few lemmas.
In Lemma \ref{lem:goodcover}, we present a nice covering for open sets.
A controlled partition of unity with respect to such a nice covering given in Lemma \ref{lem:goodcover} is constructed in Lemma \ref{lem:unity}, and we consider the associated discrete convolutions in Lemma \ref{lem:unity.approx}.
\begin{lem}\label{lem:goodcover}
    Let $(X,d)$ be a doubling metric space and let $U$ be an open set of $X$.
    Given $\lambda, A \in [1,\infty)$, there exists a family of balls $\mathcal{B} = \{ B(x_{i},r_{i}) \}_{i \in I_{\mathcal{B}}}$ with $I_{\mathcal{B}} \subset \mathbb{N}$ such that the following properties hold:
    \begin{enumerate}[label=\textup{(\alph*)},align=left,leftmargin=*,topsep=2pt,parsep=0pt,itemsep=2pt]
        \item $\{ B(x_{i}, r_{i} )\}_{i \in I_{\mathcal{B}}}$ are pairwise disjoint.
        \item $B(x_{i},\lambda r_{i}) \subset U$ for any $i \in I_{\mathcal{B}}$.
        \item $\bigcup_{i \in I_{\mathcal{B}}}B(x_{i},3r_{i}) \cap U = U$.
        \item There exists $\kappa_{1} \in [1,\infty)$ depending only on $\lambda,A$ such that for any $i,j \in I_{\mathcal{B}}$ with $B(x_{i},3Ar_{i}) \cap B(x_{j},3Ar_{j}) \cap U \neq \emptyset$,
        \begin{equation}\label{e:radcomparison}
            r_{i} \vee r_{j} \le \kappa_{1}(r_{i} \wedge r_{j}).
        \end{equation}
        \item There exists $N_{1} \in [1,\infty)$ depending only on $\lambda,A,N_{\mathrm{D}}$ such that
        \begin{equation}\label{e:finiteoverlap}
            \sum_{i \in I_{\mathcal{B}}}\indicator{B(x_{i},3Ar_{i}) \cap U} \le N_{1}.
        \end{equation}
    \end{enumerate}
\end{lem}
\begin{defn}[$(\lambda,A)$-good cover]
    A collection $\mathcal{B}$ as above in Lemma \ref{lem:goodcover} is called a \emph{$(\lambda,A)$-good cover of $U$}.
    In addition, for $\varepsilon \in (0,\infty)$, we say that $\mathcal{B} = \{ B(x_{i},r_{i}) \}_{i \in I_{\mathcal{B}}}$ is a $(\lambda,A)$-good cover of $U$ \emph{at scale $\varepsilon$} if and only if $r_{i} \le \varepsilon$ for any $i \in I_{\mathcal{B}}$.
\end{defn}
\begin{rmk}
    Such a nice covering plays important roles in several works; see, e.g., \cite{BBS07,CW71,GS11,HKT07,Mur24}.
\end{rmk}
\begin{proof}[Proof of Lemma \ref{lem:goodcover}]
    This follows from a result of Whitney-type coverings; see, e.g., \cite[Proposition 5.2]{Mur24} or Proposition \ref{prop:whitney}.
\end{proof}

\begin{lem}[Controlled partition of unity]\label{lem:unity}
    Assume that \ref{e:VD} and \ref{e:capu} hold.
    Let $A \in [A_{\mathrm{cap},1},\infty)$ and let $\lambda \in (A,\infty)$.
    Let $U$ be an open set of $X$ and let $\mathcal{B} = \{ B(x_{i},r_{i}) \}_{i \in I_{\mathcal{B}}}$ be a $(\lambda,A)$-good cover of $U$.
    Then there exists a family of functions $\{ \psi_{i} \}_{i \in I_{\mathcal{B}}}$ that satisfies the following properties:
    \begin{enumerate}[label=\textup{(\alph*)},align=left,leftmargin=*,topsep=2pt,parsep=0pt,itemsep=2pt]
        \item For any $i \in I_{\mathcal{B}}$, we have $\psi_{i} \in \mathcal{F}_{p} \cap C_{c}(X)$, $0 \le \psi_{i} \le 1$, $\psi_{i}|_{B(x_{i},3r_{i})} \ge N_{1}^{-1}$ and $\supp_{X}[\psi_{i}] \subset B(x_{i}, 3Ar_{i})$, where $N_{1}$ is the constant in \eqref{e:finiteoverlap};
        \item $\sum_{i \in I_{\mathcal{B}}}\psi_i \equiv 1$ on $U$;
        \item There is $C = C(p,c_{\mathrm{D}},C_{\mathrm{cap}},A,A_{\mathrm{cap},2},C_{\Psi},\beta_{2},\kappa_{1},N_{1}) > 0$ such that for all $i \in I_{\mathcal{B}}$,
        \begin{equation}\label{e:low-energy}
            \mathcal{E}_{p}(\psi_{i})\le C\frac{m(B(x_{i},r_{i}))}{\Psi(r_{i})}.
        \end{equation}
    \end{enumerate}
    A collection $\{ \psi_{i} \}_{i \in I_{\mathcal{B}}}$ as above is called a \emph{controlled partition of unity with respect to $\mathcal{B}$.}
\end{lem}
\begin{proof} 
    The proof is similar to \cite[Lemma 2.5]{Mur20}. 
    Recall $A \ge A_{\mathrm{cap},1}$. 
    By Lemma \ref{lem:capu.reduction}, for each $i \in I_{\mathcal{B}}$ with $r_i < \diam(X,d)/A_{\mathrm{cap},2}$, there exists $\varphi_{i} \in \mathcal{F}_{p} \cap C_{c}(X)$ such that $\varphi_{i}|_{B(x_{i},3r_{i})} = 1$, $\supp_{X}[\varphi_{i}] \subset B(x_{i},3Ar_{i})$ and $\mathcal{E}_{p}(\varphi_{i}) \le c_{1}\frac{m(B(x_{i},r_{i}))}{\Psi(r_{i})}$. 
    One can see that such $\varphi_{i} \in \mathcal{F}_{p} \cap C_{c}(X)$ exists even when $\diam(X,d) < \infty$ and $r_i \ge \diam(X,d)/A_{\mathrm{cap},2}$ by a covering argument as in the proof of Proposition \ref{prop:capu.refine}. 
    Now we define $\psi_{i} \coloneqq \varphi_{i} \cdot \bigl(\sum_{j \in I_{\mathcal{B}}}\varphi_{j}\bigr)^{-1}$, where $0/0$ is considered as $0$.
    Note that, by \eqref{e:finiteoverlap} and $\supp_{X}[\varphi_{i}] \subset B(x_{i},3Ar_{i})$, $1 \le \indicator{\supp_{X}[\varphi_{i}]}\sum_{j \in I_{\mathcal{B}}}\varphi_{j} \le N_{1}$.
    Since $\varphi_{j}|_{B(x_{i},3Ar_{i})} = 0$ whenever $B(x_{i},3Ar_{i}) \cap B(x_{j},3Ar_{j}) = \emptyset$, we observe that $\psi_{i} = \varphi_{i} \cdot (\varphi_{i} + \eta_{i})^{-1}$, where
    \[
    \eta_{i} \coloneqq \sum_{j \in I_{\mathcal{B}}; B(x_{i},3Ar_{i}) \cap B(x_{j},3Ar_{j}) \neq \emptyset}\varphi_{j} - \varphi_{i}.
    \]
    Hence $\mathcal{E}_{p}(\psi_{i})
        \le c_{2}\left(\mathcal{E}_{p}(\varphi_{i}) + \mathcal{E}_{p}(\eta_{i})\right)$ by Proposition \hyperref[it:Lip-bdd]{\ref{prop:Lipc.list}}-\ref{it:divide}.
    Now we have \eqref{e:low-energy} since 
    \begin{equation*}
        \mathcal{E}_{p}(\eta_{i})
        \le N_{1}^{p - 1}\max_{j \in I_{\mathcal{B}}; B(x_{i},3Ar_{i}) \cap B(x_{j},3Ar_{j}) \neq \emptyset}\mathcal{E}_{p}(\varphi_{j})
        \le c_{3}\frac{m(B(x_{i},r_{i}))}{\Psi(r_{i})}. \qedhere 
    \end{equation*}
\end{proof}

\begin{lem}\label{lem:unity.approx}
    Assume that \ref{e:VD} and \ref{e:capu} hold.
    Let $\delta \in (0,\infty)$, let $U$ be an open set of $X$ and let $\mathcal{B}_{\delta} = \{ B(x_{i},r_{i}) \}_{i \in I_{\mathcal{B}_{\delta}}}$ be a $(\lambda,A)$-good cover of $U$ at scale $\delta$.
    Let $\{ \psi_{i} \}_{i \in I_{\mathcal{B}_{\delta}}}$ be a controlled partition of unity with respect to $\mathcal{B}_{\delta}$.
    We define the linear operator $A_{\delta} \colon L^{p}(U,m|_{U}) \to \mathcal{F}_{p} \cap C(X)$ by\footnote{Precisely, we should write $A_{\mathcal{B}_{\delta}}$ instead of $A_{\delta}$, but we use $A_{\delta}$ for brevity.}
    \[
    A_{\delta}u \coloneqq \sum_{i \in I_{\mathcal{B}_{\delta}}}u_{B(x_{i},3 r_{i}) \cap U}\,\psi_{i}, \quad u \in L^{p}(U,m|_{U}).
    \]
    Then $\lim_{\delta \downarrow 0}\norm{A_{\delta}u - u}_{L^{p}(U)} = 0$.
\end{lem}
\begin{proof}
	The proof is a minor modification of the argument in \cite[Proof of Lemma 7.3]{MS25} (see also \cite[Lemma 5.2]{HKT07} for a similar argument), so we omit it. 
\end{proof}

With these preparations, we can show the lower bound in Theorem \ref{thm:main1}. 
\begin{thm}\label{thm:lower}
    Assume that \ref{e:VD}, \ref{e:PI}, \ref{e:capu} hold and that $\{ \rho_{\varepsilon} \}_{\varepsilon > 0}$ satisfies Assumption \ref{assum:kernel.1}-\ref{it:assum.nonlocal.lower}.
    Then there exists $C = C(p,c_{\mathrm{D}},C_{\mathrm{cap}},A_{\mathrm{cap},1},A_{\mathrm{cap},2},C_{\Psi},\beta_{2},C_{\mathrm{P}},A_{\mathrm{P}}) \in (0,\infty)$ such that for any open set $\Omega$ of $X$ and any $u \in \mathcal{F}_{p,\mathrm{loc}}(\Omega)$,
    \begin{equation}\label{e:main.lower.gen}
        \Gamma_{p}\langle u \rangle(\Omega)
        \le C\liminf_{\varepsilon \downarrow 0}\int_{\Omega}\int_{\Omega}\frac{\abs{u(x) - u(y)}^{p}}{\Psi(d(x,y))}\rho_{\varepsilon}(x,y)\,m(dx)\,m(dy).
    \end{equation}
\end{thm}
\begin{proof}
    The desired estimate \eqref{e:main.lower.gen} is obvious when
    \[
    M \coloneqq \liminf_{\varepsilon \downarrow 0}\int_{\Omega}\int_{\Omega}\frac{\abs{u(x) - u(y)}^{p}}{\Psi(d(x,y))}\rho_{\varepsilon}(x,y)\,m(dx)\,m(dy) = \infty,
    \]
    so we can assume that $M < \infty$.
    Let $\varepsilon' \in (0,1)$.
    Taking a suitable sequence $(\varepsilon_{n})_{n \in \mathbb{N}} \subset (0,\infty)$ such that $\varepsilon_{n} \downarrow 0$ as $n \to \infty$, we can assume that
    \[
    \int_{\Omega}\int_{\Omega}\frac{\abs{u(x) - u(y)}^{p}}{\Psi(d(x,y))}\rho_{\varepsilon_{n}}(x,y)\,m(dx)\,m(dy) \le M + \varepsilon' \quad \text{for any $n \in \mathbb{N}$}.
    \]
    By using \eqref{e:kernel.lower.1} or \eqref{e:kernel.lower.2}, the same argument as in \cite[Proof of (5.5)]{LPZ24} shows that for any sufficiently small $t > 0$,
    \begin{equation}\label{e:main.upper.1}
        \int_{\Omega}\int_{\Omega}\frac{\abs{u(x) - u(y)}^{p}}{\Psi(t)}\frac{\indicator{B(y,t) \cap \Omega}(x)}{m(B(y,t))}\,m(dx)\,m(dy)
        \le \frac{(M + \varepsilon')C_{\rho}}{1 - \varepsilon'}.
    \end{equation}
    Fix a small $t_{\ast} > 0$ so that \eqref{e:main.upper.1} holds with $t_{\ast}$ in place of $t$.
    Let $U \subset \Omega$ be a relatively compact open subset with $\dist(U, X \setminus \Omega) > t_{\ast}$, let $A \coloneqq A_{\mathrm{cap},1}$ and let $u^{\#} \in \mathcal{F}_{p}$ be such that $u^{\#} = u$ $m$-a.e.\ on $U(At_{\ast})$.
    Note that $U(At_{\ast})$ is also relatively compact by Proposition \ref{prop:top}.
    For each $\delta \in \bigl(0,\frac{t_\ast}{3A+27}\bigr)$, let $\{ x_{i} \}_{i \in I_{\delta}} \subset X$ be a $(2\delta)$-net of $(X,d)$. 
    Set $\mathcal{B}_{\delta} \coloneqq \{ B(x_{j}, \delta) \}_{j \in I_{\delta}}$, then it is easy to see that $\mathcal{B}_{\delta}$ is a $(2,A)$-good cover of $X$ at scale $\delta$.
    Let $\{ \psi_{j} \}_{j \in I_{\delta}}$ be a partition of unity with respect to $\mathcal{B}_{\delta}$ as described in Lemma \ref{lem:unity} and let $A_{\delta}$ be the operator in Lemma \ref{lem:unity.approx}.
    Let us observe that
    \begin{equation}\label{e:another.rep}
        (A_{\delta}u^{\#})\bigr|_{B(x_{i},3A\delta)}
        = u^{\#}_{B(x_{i},3\delta)} + \sum_{j \in I_{\delta} \setminus \{ i \}}\Bigl(u^{\#}_{B(x_{j},3\delta)} - u^{\#}_{B(x_{i},3\delta)}\Bigr)\psi_{j}|_{B(x_{i},3A\delta)}.
    \end{equation}
    By Framework \ref{frame:EPEM}-\ref{it:EMtri},\ref{it:EMtotal},\ref{it:EMslocal}, \ref{e:VD} and \eqref{e:scalefunction},
    \begin{align*}
        &\Gamma_{p}\langle A_{\delta}u^{\#} \rangle(B(x_{i},3A\delta)) \\
        &\le c_{1}\sum_{j \in I_{\delta}; B(x_{i},3A\delta) \cap B(x_{j},3A\delta) \neq \emptyset}\abs{u^{\#}_{B(x_{i},3\delta)} - u^{\#}_{B(x_{j},3\delta)}}^{p}\frac{m(B(x_{i},\delta))}{\Psi(\delta)}.
    \end{align*}
    For any $i,j \in I_{\delta}$ with $B(x_{i},3A\delta) \cap B(x_{j},3A\delta) \neq \emptyset$, by H\"{o}lder's inequality and \eqref{e:doublevar},
    \begin{equation*}
        \abs{u^{\#}_{B(x_{i},3\delta)} - u^{\#}_{B(x_{j},3\delta)}}^{p}
        \le c_{\mathrm{D}}^{5}\fint_{B(x_{i},9\delta)}\fint_{B(x_{i},9\delta)}\abs{u^{\#}(x) - u^{\#}(y)}^{p}\,m(dx)\,m(dy).
    \end{equation*}
    Noting that $\#\{ j \in I_{\delta} \mid B(x_{i},3A\delta) \cap B(x_{j},3A\delta) \neq \emptyset \} \le N$ for some $N \in \mathbb{N}$ depending only on $A,N_{\mathrm{D}}$ and that $B(x_i,9\delta) \cup B(y,18\delta) \subset U(t_\ast) \subset \Omega$ whenever $B(x_i,3A\delta) \cap U \neq \emptyset$ and $y \in B(x_i,9\delta)$, we get
    \begin{align}\label{e:em.approx}
        &\Gamma_{p}\langle A_{\delta}u^{\#} \rangle(U) 
        \le \sum_{i \in I_{\delta}; B(x_{i},3A\delta) \cap U \neq \emptyset}\Gamma_{p}\langle A_{\delta}u^{\#} \rangle(B(x_{i},3A\delta) \cap U) \nonumber \\
        &\le c_{2}\sum_{i \in I_{\delta}; B(x_{i},3A\delta) \cap U \neq \emptyset}\frac{m(B(x_{i},\delta))}{\Psi(\delta)}\fint_{B(x_{i},9\delta)}\fint_{B(x_{i},9\delta)}\abs{u(x) - u(y)}^{p}\,m(dx)\,m(dy) \nonumber \\
        &\le \frac{c_{3}}{\Psi(\delta)}\sum_{i \in I_{\delta}; B(x_{i},3A\delta) \cap U \neq \emptyset}\int_{B(x_{i},9\delta)}\int_{B(x_{i},9\delta)}\abs{u(x) - u(y)}^{p}\frac{\indicator{B(y,18\delta)}(x)}{m(B(y,18\delta))}\,m(dx)\,m(dy) \nonumber \\
        &\le c_{4}\int_{\Omega}\int_{\Omega}\frac{\abs{u(x) - u(y)}^{p}}{\Psi(18\delta)}\frac{\indicator{B(y,18\delta) \cap \Omega}(x)}{m(B(y,18\delta))}\,m(dx)\,m(dy)
        \overset{\eqref{e:main.upper.1}}{\le} \frac{c_{5}(M + \varepsilon')}{1 - \varepsilon'}.
    \end{align} 
    Similarly, we also have
    \begin{align}\label{e:discreteconv.prePI}
        &\mathcal{E}_{p}(A_{\delta}u^{\#})
        = \Gamma_{p}\langle A_{\delta}u^{\#} \rangle(X) \nonumber \quad \text{(by Framework \ref{frame:EPEM}-\ref{it:EMtotal})} \\
        &\le c_{6}\sum_{i \in I_{\delta}}\frac{m(B(x_{i},\delta))}{\Psi(\delta)}\fint_{B(x_{i},9\delta)}\fint_{B(x_{i},9\delta)}\abs{u^{\#}(x) - u^{\#}(y)}^{p}\,m(dx)\,m(dy) \\
        &\overset{\eqref{e:doublevar},\ref{e:PI}}{\le} c_{7}\sum_{i \in I_{\delta}}\Gamma_{p}\langle u^{\#} \rangle(B(x,9A_{\mathrm{P}}\delta))
        \le c_{8}\Gamma_{p}\langle u^{\#} \rangle(X) < \infty. \nonumber
    \end{align}
    Therefore, $\{ A_{\delta}u^{\#} \}_{\delta \in (0,t_{\ast}/(3A+27))}$ is bounded in $\mathcal{F}_{p}$.
    Since $\mathcal{F}_{p}$ is reflexive (Framework \ref{frame:EPEM}-\ref{it:banach}) and $\{ A_{\delta}u^{\#} \}$ converges in $L^{p}(X,m)$ to $u^{\#}$ as $\delta \downarrow 0$ (Lemma \ref{lem:unity.approx}), there exists a subsequence $\{ A_{\delta_{k}}u^{\#} \}_{k}$ of $\{ A_{\delta}u^{\#} \}_{\delta \in (0,t_{\ast}/(3A+27))}$ such that $\{ A_{\delta_{k}}u^{\#} \}_{k}$ converges weakly in $\mathcal{F}_{p}$ to $u^{\#}$.
    By Mazur's lemma (see, e.g., \cite[p.~9]{HKST}), there exist a sequence of convex combinations $\{ \sum_{k = l}^{N(l)}\lambda_{l,k}(A_{\delta_{k}}u^{\#}) \}_{l}$ that converges strongly in $\mathcal{F}_{p}$ to $u^{\#}$, which together with \eqref{e:em.approx} shows that
    \[
    \Gamma_{p}\langle u \rangle(U) = \Gamma_{p}\langle u^{\#} \rangle(U) \le \frac{c_{5}(M + \varepsilon')}{1 - \varepsilon'}.
    \]
    By exhausting $\Omega$ by sets $U$ and letting $\varepsilon' \downarrow 0$, we obtain \eqref{e:main.lower.gen}.
\end{proof} 

The following variant of \ref{e:PI} is now immediate.
\begin{prop}\label{prop:newPI}
    Assume that \ref{e:VD}, \ref{e:PI}, \ref{e:capu} hold and that $\{ \rho_{\varepsilon} \}_{\varepsilon > 0}$ satisfies Assumption \ref{assum:kernel.1}-\ref{it:assum.nonlocal.lower}.
    Then there exists $C = C(p,c_{\mathrm{D}},C_{\mathrm{cap}},A_{\mathrm{cap},1},A_{\mathrm{cap},2},C_{\Psi},\beta_{2},C_{\mathrm{P}},A_{\mathrm{P}}) \in (0,\infty)$ such that for any $x \in X$, $r > 0$ and $u \in \mathcal{F}_{p,\mathrm{loc}}(B(x,A_{\mathrm{P}}\,r))$,
    \begin{align}\label{e:PI.variant}
        &\int_{B(x,r)}\abs{u - u_{B(x,r)}}^{p}\,dm \nonumber \\
        &\quad\le C\Psi(r)\liminf_{\varepsilon \downarrow 0}\int_{B(x,A_{\mathrm{P}}r)}\int_{B(x,A_{\mathrm{P}}r)}\frac{\abs{u(y) - u(z)}^{p}}{\Psi(d(y,z))}\rho_{\varepsilon}(y,z)\,m(dy)\,m(dz). 
    \end{align}
\end{prop}
\begin{proof}
    This follows from \ref{e:PI} and Theorem \ref{thm:lower}.
\end{proof}

\subsection{Proof of Theorem \ref{thm:main1} and its corollaries}\label{subsec:proofs}
Now we prove our main result and its corollaries.

\begin{proof}[Proof of Theorem \ref{thm:main1}]
    The upper estimate in \eqref{e:main1} and \eqref{e:WM} are proved in Theorems \ref{thm:upper.ext} and \ref{thm:upper.local.KSext}, and the lower estimate in \eqref{e:main1} follows from Theorem \ref{thm:lower}.
\end{proof}

\begin{proof}[Proof of Corollary \ref{cor:main1.1}]
    Fix a sequence $\{ \theta(\varepsilon) \}_{\varepsilon > 0} \subset (0,\infty)$ with $\theta(\varepsilon) \uparrow \theta_{p}$ as $\varepsilon \downarrow 0$, and let $\bigl\{ \rho_{\varepsilon}^{\mathrm{BBM}} \bigr\}_{\varepsilon > 0}$ be the family given in \eqref{e:BBMkernel}. 
    Then we can check Assumption \ref{assum:kernel.1} for $\bigl\{ \rho_{\varepsilon}^{\mathrm{BBM}} \bigr\}_{\varepsilon > 0}$ with $\nu_{\varepsilon}(dt) \coloneqq p\theta(\varepsilon)(\theta_{p} - \theta(\varepsilon))t^{-p\theta(\varepsilon) - 1}\,dt$ similarly to \cite[Corollary 6.1]{LPZ24} and then apply Theorem \ref{thm:main1} to $\bigl\{ \rho_{\varepsilon}^{\mathrm{BBM}} \bigr\}_{\varepsilon > 0}$.
\end{proof}

\begin{proof}[Proof of Corollary \ref{cor:main1.2}]
	Let $\bigl\{ \rho_{\varepsilon}^{\mathrm{KS}} \bigr\}_{\varepsilon > 0}$ and $\bigl\{ \widehat{\rho}_{\varepsilon}^{\,\mathrm{KS}} \bigr\}_{\varepsilon > 0}$ be the families given in \eqref{e:KSkernel}. 
    Then, clearly, $\bigl\{ \rho_{\varepsilon}^{\mathrm{KS}} \bigr\}_{\varepsilon > 0}$ and $\bigl\{ \widehat{\rho}_{\varepsilon}^{\,\mathrm{KS}} \bigr\}_{\varepsilon > 0}$ satisfy \eqref{e:kernel.lower.1}.
    By Proposition \ref{prop:PI-up} and the reverse volume doubling property of $m$ (see, e.g., \cite[Exercise 13.1]{Hei}), there exist $c_{1}, Q \in (0,\infty)$ such that
    \begin{equation}\label{e:RVD}
        \frac{m(B(y,r))}{m(B(y,R))} \le c_{1}\left(\frac{r}{R}\right)^{Q} \quad \text{whenever $y \in X$ and $0 < r \le R \le \diam(X,d)$}.
    \end{equation}
    Therefore, by putting $d_{j}(\varepsilon) \coloneqq c_{1}2^{(-j+1)Q}$, we have 
    \begin{equation*}
        \rho_{\varepsilon}^{\mathrm{KS}}(x,y)
        \le \widehat{\rho}_{\varepsilon}^{\,\mathrm{KS}}(x,y)
		\le \sum_{j = 1}^{\infty}d_{j}(\varepsilon)\frac{\indicator{B(y,2^{-j+1}\varepsilon) \setminus B(y,2^{-j}\varepsilon)}(x)}{m(B(y,2^{-j}\varepsilon))}. 
    \end{equation*} 
    Since $\sum_{j = 1}^{\infty}d_{j}(\varepsilon) = c_{1}/(1 - 2^{-Q})$, both $\bigl\{ \rho_{\varepsilon}^{\mathrm{KS}} \bigr\}_{\varepsilon > 0}$ and $\bigl\{ \widehat{\rho}_{\varepsilon}^{\,\mathrm{KS}} \bigr\}_{\varepsilon > 0}$ satisfy \eqref{e:truncate.upper.2} and \eqref{e:truncate.upper.const}.
    Note that \eqref{e:kernel.compl} for $\bigl\{ \rho_{\varepsilon}^{\mathrm{KS}} \bigr\}_{\varepsilon > 0}$ and for $\bigl\{ \widehat{\rho}_{\varepsilon}^{\,\mathrm{KS}} \bigr\}_{\varepsilon > 0}$ is obvious since $\rho_{\varepsilon}^{\mathrm{KS}}(x,y) = \widehat{\rho}_{\varepsilon}^{\,\mathrm{KS}}(x,y) = 0$ whenever $d(x,y) \ge \varepsilon$.
    We complete the proof by applying Theorem \ref{thm:main1} to these families.
\end{proof}

\begin{proof}[Proof of Corollary \ref{cor:main1.3}]
	We only show the first equality in \eqref{e:determine} because the other one is similar. 
    Set $\mathcal{D}_{p}^{\Psi,\rho} \coloneqq \bigl\{ u \in L^{p}(X,m) \bigm| \liminf_{\varepsilon \downarrow 0}\int_{X}\int_{X}\frac{\abs{u(x) - u(y)}^{p}}{\Psi(d(x,y))}\rho_{\varepsilon}(x,y)\,m(dx)\,m(dy) < \infty \bigr\}$. 
    By Theorem \ref{thm:main1}, $\mathcal{F}_{p} \subset \mathcal{D}_{p}^{\Psi,\rho}$ is clear.
    Let $u \in \mathcal{D}_{p}^{\Psi,\rho}$.
    As in the proof of Theorem \ref{thm:lower}, for each $\delta > 0$, let $\{ x_{i} \}_{i \in I_{\delta}} \subset X$ be a $(2\delta)$-net of $(X,d)$, let $\mathcal{B}_{\delta} \coloneqq \{ B(x_{j}, \delta) \}_{j \in I_{\delta}}$, which is a $(2,A_{\mathrm{cap},1})$-good cover of $X$ at scale $\delta$, let $\{ \psi_{j} \}_{j \in I_{\delta}}$ be a partition of unity with respect to $\mathcal{B}_{\delta}$ as described in Lemma \ref{lem:unity}, and let $A_{\delta}u \in \mathcal{F}_{p} \cap C(X)$ be the discrete convolution of $u$ as in Lemma \ref{lem:unity.approx}. 
    Then, by \eqref{e:discreteconv.prePI}, \eqref{e:doublevar} and \eqref{e:PI.variant}, we can see that 
    \[
    \mathcal{E}_{p}(A_{\delta}u) 
    \le c_{1}\liminf_{\varepsilon \downarrow 0}\int_{X}\int_{X}\frac{\abs{u(y) - u(z)}^{p}}{\Psi(d(y,z))}\rho_{\varepsilon}(y,z)\,m(dy)\,m(dz) < \infty,
    \]
    whence it follows that $\{ A_{\delta}u \}_{\delta > 0}$ is bounded in $\mathcal{F}_{p}$.
    By the reflexivity of $\mathcal{F}_{p}$ (Framework \ref{frame:EPEM}-\ref{it:banach}), there exists a subsequence $\{ A_{\delta_{k}}u \}_{k}$ that converges weakly in $\mathcal{F}_{p}$ to some $u_{\ast} \in \mathcal{F}_{p}$.
    Since $\lim_{\delta \downarrow 0}\norm{A_{\delta}u - u}_{L^{p}(X)} = 0$ by Lemma \ref{lem:unity.approx}, we get $u = u_{\ast}$, proving $\mathcal{D}_{p}^{\Psi,\rho} \subset \mathcal{F}_{p}$.
\end{proof}

\section{$L^{p}$-Besov critical exponents}\label{sec:critical}
In this section, we will see that the \emph{$L^{p}$-Besov critical exponent} turns out to be equal to $\beta/p$ under the assumptions \ref{e:VD}, \hyperref[e:PI]{\textup{PI$_p$($\beta$)}} and \hyperref[e:capu]{\textup{cap$_p$($\beta$)$_{\le}$}}.

We first recall the definition of Besov-type spaces on $(X,d,m)$.
\begin{defn}\label{defn:besov}
    For $\theta > 0$, we define $B^{\theta}_{p,p}(X)$ and $B^{\theta}_{p,\infty}(X)$ by
    \begin{equation*}
        B^{\theta}_{p,p}(X)
        \coloneqq
        \biggl\{ u \in L^{p}(X,m) \biggm| \int_{X}\int_{X}\frac{\abs{u(x) - u(y)}^{p}}{d(x,y)^{p\theta}m(B(y,d(x,y)))}\,m(dx)\,m(dy) < \infty \biggr\},
    \end{equation*}
    and $B^{\theta}_{p,\infty}(X) \coloneqq \bigl\{ u \in L^{p}(X,m) \bigm| \sup_{r > 0}E_{p,\theta}(u,r) < \infty \bigr\}$, 
   where 
   \[
   E_{p,\theta}(u,r) \coloneqq \int_{X}\fint_{B(y,r)}\frac{\abs{u(x) - u(y)}^{p}}{r^{p\theta}}\,m(dx)\,m(dy). 
   \]
\end{defn}
\begin{rmk}
    Under \ref{e:VD}, we have $B^{\theta}_{p,p}(X) = \bigl\{ u \in L^{p}(X,m) \bigm| \int_{0}^{\diam(X)}E_{p,\theta}(u,r)\,\frac{dr}{r} < \infty \bigr\}$. 
    Indeed, we know from \cite[Theorem 5.2]{GKS10} that
    \begin{equation*}
        \int_{0}^{\diam(X)}E_{p,\theta}(u,r)\,\frac{dr}{r} 
        \quad \asymp \int_{X}\int_{X}\frac{\abs{u(x) - u(y)}^{p}}{d(x,y)^{p\theta}m(B(y,d(x,y)))}\,m(dx)\,m(dy),
    \end{equation*}
    where, by $A \asymp B$ we mean that that there is an implicit constant $C \ge 1$ depending on some unimportant parameters such that $C^{-1} \le A/B \le C$.
\end{rmk}

The following proposition collects basic properties of $B^{\theta}_{p,q}(X)$, $q \in \{ p,\infty \}$.
\begin{prop}\label{prop:basic.besov}
    \begin{enumerate}[label=\textup{(\alph*)},align=left,leftmargin=*,topsep=2pt,parsep=0pt,itemsep=2pt]
        \item\label{it:besov.difftheta} $B^{\theta}_{p,\infty}(X) \subset B^{\theta'}_{p,\infty}(X)$ whenever $0 < \theta' < \theta$.
        \item\label{it:besov.KS} It holds that $B^{\theta}_{p,\infty}(X) = \bigl\{ u \in L^{p}(X,m) \bigm| \limsup_{r \downarrow 0}E_{p,\theta}(u,r) < \infty \bigr\}$. 
        \item\label{it:besov.relation} For any $\theta > 0$ and $\delta \in (0,\theta)$,
        \begin{equation}\label{e:relation.Besov}
            B^{\theta}_{p,p}(X) \subset B^{\theta}_{p,\infty}(X) \subset B^{\theta - \delta}_{p,p}(X).
        \end{equation}
    \end{enumerate}
\end{prop}
%
%
\begin{proof}
    \ref{it:besov.difftheta}, \ref{it:besov.KS} and \ref{it:besov.relation} are proved in \cite[Corollary 3.3, Lemma 3.2]{Bau24} and \cite[Proposition 2.2]{GYZ23+} respectively. See also \cite[Lemma 2.5]{KSS24+}. 
\end{proof}

Similar to \cite[Section 5.2]{ABCRST1}, we define critical exponents $\alpha_{p}$ and $\alpha_{p}^{\ast}$ as follows.
We also introduce variants of these critical exponents as in \cite{KSS24+}.  
\begin{defn}[$L^{p}$-Besov critical exponents]\label{defn:criticalexponents}
    We define
    \begin{align*}
        \alpha_{p}(X) &\coloneqq \sup\{ \theta > 0 \mid \text{$B^{\theta}_{p,\infty}(X)$ contains a non-constant function} \}, \\
        \alpha_{p}^{\ast}(X) &\coloneqq \sup\{ \theta > 0 \mid \text{$B^{\theta}_{p,\infty}(X)$ is dense in $L^{p}(X,m)$} \}, \\
        \alpha_{p,p}(X) &\coloneqq \sup\{ \theta > 0 \mid \text{$B^{\theta}_{p,p}(X)$ contains a non-constant function} \}, \\
        \alpha_{p,p}^{\ast}(X) &\coloneqq \sup\{ \theta > 0 \mid \text{$B^{\theta}_{p,p}(X)$ is dense in $L^{p}(X,m)$} \}. 
    \end{align*}
\end{defn}

The following proposition gives a relation among these critical exponents, whose proof is omitted because it is clear from the definition and \eqref{e:relation.Besov}. 
\begin{prop}\label{prop:exponents.same}
    Assume that \ref{e:VD} holds. 
    Then $\alpha_{p,p}^{\ast}(X) = \alpha_{p}^{\ast}(X) \le \alpha_{p}(X) = \alpha_{p,p}(X)$. 
\end{prop}
%

Now we prove the main result in this section.
\begin{thm}
    Assume that \ref{e:VD}, \hyperref[e:PI]{\textup{PI$_{p}$($\beta$)}} and \hyperref[e:capu]{\textup{cap$_{p}$($\beta$)$_{\le}$}} hold for some $\beta \in (0,\infty)$.
    Then $\alpha_{p}(X) = \alpha_{p}^{\ast}(X) = \alpha_{p,p}(X) = \alpha_{p,p}^{\ast}(X) = \beta/p$.
\end{thm}
\begin{proof}
    By Corollaries \ref{cor:main1.2} and \ref{cor:main1.3}, we know that $\mathcal{F}_{p} = B^{\beta/p}_{p,\infty}(X)$.
    Since $\mathcal{F}_{p}$ is dense in $L^{p}(X,m)$ by Framework \ref{frame:EPEM}-\ref{it:regular}, we obtain $\alpha_{p}^{\ast}(X) \ge \beta/p$.
    Let us fix $\theta > \beta/p$ and $u \in \mathcal{F}_{p}\setminus\mathbb{R}\indicator{X}$.
    Note that $\mathcal{E}_{p}(u) > 0$ (recall that $\mathcal{E}_{p}^{-1}(0) \subset \mathbb{R}\indicator{X}$ by Framework \ref{frame:EPEM}-\ref{it:Lip}) and that $B^{\theta}_{p,\infty}(X)\setminus\mathbb{R}\indicator{X} \subset \mathcal{F}_{p}\setminus\mathbb{R}\indicator{X}$.
    Then we have $r^{p(\theta - \beta/p)}E_{p,\theta}(u,r) \ge \inf_{s \in (0,r]}E_{p,\beta/p}(u,s)$ for any $r > 0$. 
    Since $ \lim_{r \downarrow 0}\inf_{s \in (0,r]}E_{p,\beta/p}(u,s) \ge c_{1}\mathcal{E}_{p}(u) > 0$ by \eqref{e:KS}, we conclude that $\liminf_{r \downarrow 0}E_{p,\theta}(u,r) = \infty$, from which it follows that $B^{\theta}_{p,\infty}(X)\setminus\mathbb{R}\indicator{X} = \emptyset$.
    In particular, $\alpha_{p}(X) \le \beta/p$.
    Since $\alpha_{p}^{\ast}(X) \le \alpha_{p}(X)$, $\alpha_{p,p}(X) = \alpha_{p}(X)$ and $\alpha_{p,p}^{\ast}(X) = \alpha_{p}^{\ast}(X)$ by Proposition \ref{prop:exponents.same}, we complete the proof.
\end{proof}

Let us conclude this section by showing an estimate for the critical exponent $\alpha_{p}(X)$ (Theorem \ref{thm:exp.upper} below), which is an analogue of \cite[Theorem 4.8(ii)]{GHL03} and an improvement of \cite[Theorem 4.2]{Bau24}.
By considering a distance function, one can show $\alpha_{p}(X) \ge 1$ (see \cite[Theorem 4.1]{Bau24}), so we are interested in an upper bound for $\alpha_{p}(X)$.
To state Theorem \ref{thm:exp.upper}, we need to recall the \emph{chain condition} and the \emph{Ahlfors regularity} (see, e.g., \cite[Definition 3.4 and (3.2)]{GHL03}).
\begin{defn}
    \begin{enumerate}[label=\textup{(\arabic*)},align=left,leftmargin=*,topsep=2pt,parsep=0pt,itemsep=2pt]
        \item The metric space $(X,d)$ is said to satisfy the \emph{chain condition} if and only if there exists $C \in (0,\infty)$ such that the following condition holds: for any $n \in \mathbb{N}$ and $x,y \in X$ with $x \neq y$, there exists $\{ x_i \}_{i = 0}^{n} \subset X$ such that $x_{0} = x$, $x_{n} = y$ and
        \[
        d(x_{i - 1}, x_{i}) \le C\frac{d(x,y)}{n} \quad \text{for any $i \in \{  1,\dots,n \}$.}
        \]
        \item Let $Q \in (0,\infty)$. The measure $m$ is said to be \emph{$Q$-Ahlfors regular} if and only if there exists $C_{\mathrm{AR}} \in [1,\infty)$ such that for any $x \in X$ and $r \in (0,2\diam(X,d))$,
        \[
        C_{\mathrm{AR}}^{-1}\,r^Q \le m(B(x,r)) \le C_{\mathrm{AR}}\,r^Q.
        \]
    \end{enumerate}
\end{defn}

\begin{thm}\label{thm:exp.upper}
    Assume that $(X,d)$ satisfies the chain condition and that $m$ is $Q$-Ahlfors regular for some $Q \in (0,\infty)$.
    Then $\alpha_{p}(X) \le \frac{Q + p - 1}{p}$.
\end{thm}
\begin{proof}
    Let $\theta > \frac{Q + p - 1}{p}$ and $u \in B_{p,\infty}^{\theta}(X)$.
    Note that the Lebesgue differentiation theorem holds on $(X,d,m)$ (see, e.g., \cite[p.~77]{HKST}).
    Let $z,z' \in X$ be Lebesgue points of $u$ with $z \neq z'$.
    By the chain condition, for any $n \in \mathbb{N}$ there exists $\{ x_i \}_{i = 0}^{n} \subset X$ such that $x_{0} = z$, $x_{n} = z'$ and $d(x_{i - 1},x_{i}) \le C\frac{d(z,z')}{n} \eqqcolon \rho_{n}$.
    Let us consider large enough $n \in \mathbb{N}$ so that $2\rho_{n} < d(z,z')$.
    By \cite[Lemma 4.10]{GHL03}, we can pick a subsequence $\{ x_{i_{k}} \}_{k = 0}^{l}$ of $\{ x_{i} \}_{i = 0}^{n}$ so that $l \le n$, $x_{i_{0}} = z$, $x_{i_{l}} = z'$, $\{ B(x_{i_{k}}, \rho_{n}) \}_{k}$ are disjoint and $d(x_{i_{k - 1}},x_{i_{k}}) < 5\rho_{n}$.
    Set $B_{k,n} \coloneqq B(x_{i_{k}}, \rho_{n})$ and $r_{n} \coloneqq 7\rho_{n}$ as in \cite[Proof of Theorem 4.8(ii)]{GHL03}.
    Noting that $\{ (x,y) \mid x \in B_{k - 1,n}, y \in B_{k,n} \} \subset \{ (x,y) \mid y \in B_{k,n}, x \in B(y,r_{n}) \}$, we see that
    \begin{align*}
        E_{p,\theta}(u,r_{n}) 
        &\ge C_{\mathrm{AR}}^{-1}\,r_{n}^{-p\theta - Q}\sum_{k = 1}^{l}\int_{B_{k,n}}\int_{B_{k-1,n}}\abs{u(x) - u(y)}^{p}\,m(dx)\,m(dy) \\
        &\ge C_{\mathrm{AR}}^{-1}\,r_{n}^{-p\theta - Q}\sum_{k = 1}^{l}m(B_{k-1,n})m(B_{k,n})\abs{u_{B_{k-1}} - u_{B_{k}}}^{p} \\
        &\ge c_{1}C_{\mathrm{AR}}^{-3}\,r_{n}^{-p\theta + Q}\frac{1}{l^{p - 1}}\abs{\sum_{k = 1}^{l}\abs{u_{B_{k - 1,n}} - u_{B_{k,n}}}}^{p}
        \ge c_{2}\,r_{n}^{-p\theta + Q + p - 1}\frac{\abs{u_{B_{0,n}} - u_{B_{l,n}}}^{p}}{d(z,z')^{p - 1}},
    \end{align*}
    where we used H\"{o}lder's inequality in the second and third inequalities.
    Since $p\theta - Q - p + 1 > 0$, we have
    \begin{align*}
        \abs{u(z) - u(z')}^{p}
        = \lim_{n \to \infty}\abs{u_{B_{0,n}} - u_{B_{l,n}}}^{p} 
        \le c_{2}^{-1}d(z,z')^{p - 1}\biggl(\sup_{r > 0}E_{p,\theta}(u,r)\biggr)
        \lim_{n \to \infty}r_{n}^{p\theta - Q - p + 1} = 0,
    \end{align*}
    whence it follows that $u$ is a constant function.
    This shows $\alpha_{p}(X) \le \frac{Q + p - 1}{p}$.
\end{proof}

\section{The extension property \hyperref[defn:sextdomain]{\textup{(E)}} for uniform domains}\label{sec.ud}
The aim of this section is to prove Theorem \ref{thm:ud.sext}, which can be done with straightforward modifications of the arguments in \cite[Section 5]{Mur24}. 
Most results in this section are analogues of those in \cite[Section 5]{Mur24}. More precisely, the $(2,2)$-Poincar\'e inequality on uniform domains and some $L^2$-estimates for extension maps (see \eqref{e:defn.extmap}) are shown in \cite[Theorem 5.3 and Proposition 5.8]{Mur24} under the two-sided heat kernel estimate HKE$(\Psi)$ (see \cite[Definition 2.6 and Assumption 5.5]{Mur24}). By noting that Corollaries \ref{cor:main1.2} and \ref{cor:main1.3} can be used instead of \cite[Theorem 4.6-(b),(c)]{Mur24}, the condition HKE$(\Psi)$ in the results above can be weakened to the conjunction of \ref{e:VD}, \hyperref[e:PI]{\textup{PI$_2$($\Psi$)}} and \hyperref[e:capu]{\textup{cap$_2$($\Psi$)$_{\le}$}}. We aim to prove natural $p$-analogues of them by replacing \hyperref[e:PI]{\textup{PI$_2$($\Psi$)}} and \hyperref[e:capu]{\textup{cap$_2$($\Psi$)$_{\le}$}} with \ref{e:PI} and \ref{e:capu} respectively. 
We record sketches of these modifications and some related results for potential future applications.

\subsection{Whitney covering, uniform domain and their properties}
We first recall the definition of uniform domains, which was originally introduced by Martio and Sarvas in \cite{MS79}.
Here we employ the version of the definition of uniform domains due to \cite{Mur24}; 
see \cite[Section 2.2]{Mur24} for it's relation with the definition in the literature and for some concrete examples of uniform domains.
\begin{defn}[{\cite[Definition 2.3]{Mur24}}]
    Let $A \ge 1$.
    A connected, non-empty open set $U \subset X$ with $U \neq X$ is called a \emph{$A$-uniform domain} if and only if for every $x,y \in U$ there exists a curve $\gamma$ in $U$ from $x$ to $y$, i.e., $\gamma \colon [a,b] \to U$ is a continuous map satisfying $\gamma(a) = x$ and $\gamma(b) = y$, such that $\diam(\gamma) \le Ad(x,y)$ and for any $z \in \gamma \coloneqq \gamma([a,b])$,
    \[
    \delta_{U}(z) \coloneqq d(x,X \setminus U) \ge A^{-1}\min\{ d(x,z), d(y,z) \}.
    \]
    Such a curve $\gamma$ is called a \emph{$A$-uniform curve}.
    We simply say that $U$ is a \emph{uniform domain} if and only if it is a \emph{$A$-uniform domain} for some $A \ge 1$.
\end{defn}

In the next propositions, we recall basic geometric properties of uniform domains.
We use the following notation; for $U \subset X$, $x \in X$ and $r > 0$,
\[
B_{U}(x,r) \coloneqq B(x,r) \cap U.
\]
\begin{prop}[{\cite[Theorem 2.8]{BS07}} or {\cite[Lemma 3.5]{Mur24}}]\label{prop:geom.ud}
    Let $U \subset X$ be a $A$-uniform domain for some $A \ge 1$.
    \begin{enumerate}[label=\textup{(\alph*)},align=left,leftmargin=*,topsep=2pt,parsep=0pt,itemsep=2pt]
        \item\label{it:corkskrew} \textup{(Corkscrew condition)} For any $x \in \closure{U}^{X}$ and $r > 0$ with $U \setminus B(x,r) \neq \emptyset$, there exists $y \in U$ such that $B(y,r/(3A)) \subset B_{U}(x,r)$.
        \item\label{it:VD.ud} Assume that \ref{e:VD} holds. Then $m(\partial U) = 0$ and there exists $C = C(c_{\mathrm{D}},A) \in [1,\infty)$ such that
        \begin{equation}\label{e:VD.ud}
            m(B_{U}(x,2r)) \le Cm(B_{U}(x,r)) \quad \text{for any $x \in X$ and $r > 0$.}
        \end{equation}
    \end{enumerate}
\end{prop}

Now we recall the notion of an $\varepsilon$-Whitney cover (see \cite[Definition 3.16]{GS11} and \cite[Definition 3.1]{Mur24}) whose existence follows from Zorn's lemma (\cite[Proposition 3.2(a)]{Mur24}), and record its basic properties.
\begin{defn}
    Let $\varepsilon \in (0,1/2)$ and let $U \subset X$ be a non-empty open set with $U \neq X$.
    A collection of balls $\mathfrak{R} \coloneqq \{ B_{U}(x_i,r_i) \mid x_i \in U, r_i > 0, i \in I \}$ is said to be an \emph{$\varepsilon$-Whitney cover} if and only if it satisfies the following properties:
    \begin{enumerate}[label=\textup{(\arabic*)},align=left,leftmargin=*,topsep=2pt,parsep=0pt,itemsep=2pt]
        \item $\{ B_{U}(x_i,r_i) \}_{i \in I}$ are pairwise disjoint.
        \item $r_{i} = \frac{\varepsilon}{1 + \varepsilon}\delta_{U}(x_i)$ for any $i \in I$. 
        \item $\bigcup_{i \in I}B_{U}(x_i,K_{\varepsilon}r_i) = U$, where $K_{\varepsilon} \coloneqq 2(1 + \varepsilon) \in (2,3)$.
    \end{enumerate}
\end{defn}

\begin{prop}[{\cite[Proposition 3.2, Lemma 3.7]{Mur24}}]\label{prop:whitney}
    Let $\varepsilon \in (0,1/2)$, let $U \subset X$ be a non-empty open set with $U \neq X$ and let $\mathfrak{R} = \{ B_{U}(x_i,r_i) \mid x_i \in U, r_i > 0, i \in I \}$ be an $\varepsilon$-Whitney cover of $U$.
    \begin{enumerate}[label=\textup{(\alph*)},align=left,leftmargin=*,topsep=2pt,parsep=0pt,itemsep=2pt]
        \item\label{it:whitney.radcomparison} \textup{(radius comparison)} For any $\lambda > 1$ with $(\lambda - 1)\varepsilon < 1$ and any $i,j \in I$ such that $B_{U}(x_i,\lambda r_i) \cap B_{U}(x_j,\lambda r_j) \neq \emptyset$, we have
        \begin{equation}\label{e:rad.comparison1}
            \frac{1 - (\lambda - 1)\varepsilon}{1 + (\lambda + 1)\varepsilon}r_{j} \le r_{i} \le \frac{1 + (\lambda + 1)\varepsilon}{1 - (\lambda - 1)\varepsilon}r_{j}. 
        \end{equation}
        \item\label{it:whitney.overlap} \textup{(bounded overlap)} Assume that $(X,d)$ is a doubling metric space. Then there exists $C = C(N_{\mathrm{D}},\varepsilon)$ such that
        \begin{equation}\label{e:whitney.finiteoverlap}
            \sum_{i \in I}\indicator{B_{U}(x_i,r_i/\varepsilon)} \le C.
        \end{equation}
        \item\label{it:whitney.central} \textup{(central ball)} Assume that $U$ is a $A$-uniform domain for some $A \ge 1$ and that $\varepsilon \in (0,1/14)$. For any $x \in \closure{U}^{X}$ and $r \in [\delta_{U}(x),\diam(U)/2)$,
        \begin{equation}\label{e:whitney.nearball}
            B_{U}(x,r) \subset \bigcup_{B_{U}(x_i,r_i) \in \mathfrak{R}(x,r)}B_{U}(x_i,3r_i) \subset B_{U}(x,2r),
        \end{equation}
        where $\mathfrak{R}(x,r) \coloneqq \{ B_{U}(x_i,r_i) \in \mathfrak{R} \mid B_{U}(x_i,3r_i) \cap B_{U}(x,r) \neq \emptyset \}$.
        In addition, there exists $B_{U}(x_0,r_0) \in \mathfrak{R}(x,r)$ such that
        \begin{equation}\label{e:centralball.rad}
            \frac{\varepsilon}{3A(4 + \varepsilon)}r \le r_0 \le \frac{2\varepsilon}{1 - 2\varepsilon}r.
        \end{equation}
    \end{enumerate}
\end{prop}
\begin{rmk}\label{rmk:whitney-good}
    By the properties \ref{it:whitney.radcomparison} and \ref{it:whitney.overlap}, for any $1 \le A' \le 1/(3\varepsilon)$, an $\varepsilon$-Whitney cover $\mathfrak{R}$ of $U$ turns out to be a $\bigl(\frac{1 + \varepsilon}{\varepsilon},A'\bigr)$-good cover of $U$.
\end{rmk}

\subsection{\texorpdfstring{$(p,p)$}{(p,p)}-Poincar\'{e} inequality on uniform domains}
Let us recall the following lemma in \cite{Mur24}, which provides us a nice chain of balls. See also \cite[Lemma 3.23]{GS11}.
\begin{lem}[{\cite[Lemma 3.9]{Mur24}}]\label{lem:chain.PI}
    Let $A \ge 1$ and $\varepsilon \in (0,1/14)$.
    Let $U \subset X$ be a $A$-uniform domain, let $\mathfrak{R} = \{ B_{U}(x_i,r_i) \mid x_i \in U, r_i > 0, i \in I \}$ be an $\varepsilon$-Whitney cover of $U$, let $x \in \closure{U}^{X}$ and let $r \in [\delta_{U}(x),\diam(U)/2)$.
    Let $B_{0} = B_{U}(x(B_{0}),r(B_{0})) \in \mathfrak{R}(x,r)$ satisfy $r(B_{0}) \ge c_{0}r$ for some $c_{0} \in (0,\infty)$, let $D = B_{U}(x(D),r(D)) \in \mathfrak{R}(x,r)$ and let $\gamma$ be a $A$-uniform curve from $x_0$ to $x(D)$.
    There exists a finite collection of distinct balls $\mathbb{S}(D) = \{ B_{0}^{D},B_{1}^{D}, \dots, B_{l}^{D} \}$ of length $l = l(D)$ such that $B_{0}^{D} = B_{0}$, $B_{l}^{D} = D$ and
    \begin{equation}\label{e:chainball}
        B_{j}^{D} = B_{U}(x_{j}(D),r_{j}(D)) \in \mathfrak{R}, \quad
        3B_{j - 1}^{D} \cap 3B_{j}^{D} \neq \emptyset, \quad
        3B_{j}^{D} \cap \gamma \neq \emptyset
    \end{equation}
    for any $j \in \{ 1,\dots,l \}$.
    Here $3B_{j}^{D} \coloneqq B_{U}(x_{j}(D),3r_{j}(D))$.
    Moreover, there exist $C_{0},C_{1} \in (0,\infty)$ depending only on $A,\varepsilon,c_{0}$ such that for any $j \in \{ 0,1,\dots,l \}$,
    \begin{equation}\label{e:chainball.rad}
        r_{j}(D) \le \frac{(A(4\varepsilon+1) + 1 - 2\varepsilon)\varepsilon}{(1 -2\varepsilon)^{2}}r, \quad
        B_{U}(x_{j}(D),r_{j}(D)) \subset B_{U}(x, C_{0}r),
    \end{equation}
    and
    \begin{equation}\label{e:chainball.center}
        D \subset B_{U}(x_{j}(D), C_{1}r_{j}(D)).
    \end{equation}
\end{lem}

The following lemma is known as the \emph{Bojarski lemma} \cite[Lemma 4.2]{Boj88}; see also \cite[Lemma 3.25]{GS11}, \cite[Exercise 2.10]{Hei}, \cite[Lemma 5.1]{Mur24}.
\begin{lem}\label{lem:Lp.overlap}
    Assume that \ref{e:VD} holds.
    Let $\{ B(x_i,r_i) \}_{i \in I}$ be a countable collection of balls and let $a_{i} \ge 0$ for $i \in I$.
    For any $q \in (1,\infty)$ and $\lambda \ge 1$, there exists $C \ge 1$ depending only on $q,\lambda$ such that
    \[
    \int_{X}\left(\sum_{i \in I}a_{i}\indicator{B(x_{i},\lambda r_{i})}\right)^{q}\,dm
    \le C\int_{X}\left(\sum_{i \in I}a_{i}\indicator{B(x_{i},r_{i})}\right)^{q}\,dm.
    \]
\end{lem}

The following lemma is an analogue of \cite[Lemma 5.2]{Mur24}.
\begin{lem}\label{lem:whitney.avediff}
    Assume that \ref{e:VD} and \ref{e:PI} hold.
    Let $U$ be a $A$-uniform domain for some $A \ge 1$, let $\varepsilon \in (0,1/4)$ be such that $A_{\mathrm{P}}\bigl(3 + \frac{6(1 + 4\varepsilon)}{1 - 2\varepsilon}\bigr)\frac{\varepsilon}{1 + \varepsilon} < 1$, where $A_{\mathrm{P}} \ge 1$ is the constant in \ref{e:PI}, and let $\mathfrak{R} = \{ B_{U}(x_i,r_i) \mid x_i \in U, r_i > 0, i \in I \}$ be an $\varepsilon$-Whitney cover of $U$.
    Then there exists $C = C(p,c_{\mathrm{D}},A,\varepsilon,C_{\Psi},\beta_{2},C_{\mathrm{P}}) \in (0,\infty)$ such that for any $i,j \in I$ with $B_{U}(x_i,3r_i) \cap B_{U}(x_j,3r_j) \neq \emptyset$ and $u \in \mathcal{F}_{p,\mathrm{loc}}(B(x_i,A_{\mathrm{P}}Lr_{i}))$, where $L \coloneqq \bigl(3 + \frac{6(1 + 4\varepsilon)}{1 - 2\varepsilon}\bigr) \le 27$,
    \[
    \abs{u_{B_{U}(x_i,3r_i)} - u_{B_{U}(x_j,3r_j)}}^{p}
    \le C\frac{\Psi(r_i)}{m(B(x_i,r_i))}\int_{B(x_i,A_{\mathrm{P}}Lr_{i})}\,d\Gamma_{p}\langle u \rangle.
    \]
\end{lem}
\begin{proof}
    Note that $B_{U}(x_{j},3r_{j}) \subset B_{U}(x_i,Lr_{i})$ by \eqref{e:rad.comparison1} and that $B_{U}(x_{i},A_{\mathrm{P}}Lr_{i}) = B(x_{i},A_{\mathrm{P}}Lr_{i})$ by $A_{\mathrm{P}}L\frac{\varepsilon}{1 + \varepsilon} < 1$.
    By \ref{e:VD} and \eqref{e:rad.comparison1}, we know that $m(B(x_i,r_i)) \le c_{1}m(B(x_j,r_j))$ and $m(B_{U}(x_i,Lr_i)) \le c_{1}m(B(x_i,r_i))$ for some $c_{1}$ depending only on $c_{\mathrm{D}},A,\varepsilon$.
    Now we see that
    \begin{align*}
        \abs{u_{B_{U}(x_i,3r_i)} - u_{B_{U}(x_j,3r_j)}}^{p}
        &\le \fint_{B_{U}(x_i,3r_i)}\fint_{B_{U}(x_j,3r_j)}\abs{u(x) - u(y)}^{p}\,m(dx)\,m(dy) \\
        &\overset{\eqref{e:doublevar},\ref{e:PI}}{\le} \frac{2^{p}c_{1}^{2}\,C_{\mathrm{P}}\Psi(Lr_i)}{m(B(x_i,r_i))}\int_{B(x_i,A_{\mathrm{P}}Lr_i)}\,d\Gamma_{p}\langle u \rangle.
    \end{align*}
    where we used H\"{o}lder's inequality in the first line.
\end{proof}

Now we can estimate discrete $p$-energy forms as follows. 
\begin{lem}\label{lem:telescope}
    Assume that \ref{e:VD} and \ref{e:PI} hold.
    Let $U$ be a $A$-uniform domain for some $A \ge 1$, let $\varepsilon \in (0,1/4)$ be such that $A_{\mathrm{P}}\bigl(3 + \frac{6(1 + 4\varepsilon)}{1 - 2\varepsilon}\bigr)\frac{\varepsilon}{1 + \varepsilon} < 1$, where $A_{\mathrm{P}} \ge 1$ is the constant in \ref{e:PI}, and let $\mathfrak{R} = \{ B_{U}(x_i,r_i) \mid x_i \in U, r_i > 0, i \in I \}$ be an $\varepsilon$-Whitney cover of $U$.
    Then for any $a_{0} \in (0,\infty)$, there exist $C,\lambda \in [1,\infty)$ (depending only on $p,a_{0},c_{\mathrm{D}},A,\varepsilon,C_{\Psi},\beta_{2},C_{\mathrm{P}},A_{\mathrm{P}}$) such that for any $x \in \closure{U}^{X}$, $r \in [\delta_{U}(x),\diam(U)/2)$, $B_{0} = B_{U}(x_{0},r_{0}) \in \mathfrak{R}(x,r)$ with $r_{0} \ge a_{0}r$ and $u \in \mathcal{F}_{p,\mathrm{loc}}(B_{U}(x,\lambda r))$,
    \begin{equation}\label{e:telescope}
        \sum_{i \in I; B_{U}(x_i,r_i) \in \mathfrak{R}(x,r)}\abs{u_{B_{U}(x_{i},3r_{i})} - u_{3B_{0}}}^{p}m(B_{U}(x_{i},3r_{i}))
        \le C\Psi(r)\int_{B_{U}(x,\lambda r)}\,d\Gamma_{p}\langle u \rangle.
    \end{equation}
\end{lem}
\begin{proof} 
    Let us fix $D \coloneqq B_{U}(x_{i},r_{i}) \in \mathfrak{R}(x,r)$ and choose a collection $\mathbb{S}(D) = \{ B_{0}^{D}, B_{1}^{D}, \dots, B_{l(D)}^{D} \}$ as given in Lemma \ref{lem:chain.PI}.
    Set $3B_{j}^{D} \coloneqq B_{U}(x_{j}(D),3r_{j}(D))$ and
    \[
    a_{j} \coloneqq \left(\frac{\Psi(r_{j}(D))}{m(B_{U}(x_{j}(D),r_{j}(D)))}\int_{B_{U}(x_{j}(D),A_{\mathrm{P}}Lr_{j}(D))}\,d\Gamma_{p}\langle u \rangle\right)^{1/p},
    \]
    where $L$ is the constant in Lemma \ref{lem:whitney.avediff}.
    By Lemma \ref{lem:chain.PI}, we know that $D \subset B_{U}(x_{j}(D),C_{1}r_{j}(D))$ for any $j \in \{ 0,1,\dots,l(D) \}$, where $C_{1}$ is the constant in \eqref{e:chainball.center}, and by Lemma \ref{lem:whitney.avediff}, 
    \begin{equation}\label{e:telescope.basis}
        \abs{u_{3B_{0}} - u_{3B_{l(D)}^{D}}}\indicator{B_{l(D)}^{D}}
        \le c_{1}\sum_{B_{U}(x_j,r_j) \in \mathfrak{R}_{\mathrm{chain}}(x,r)}a_{j}\indicator{D}\indicator{B_{U}(x_{j},C_{1}r_{j})}, 
    \end{equation}
    where $\mathfrak{R}_{\mathrm{chain}}(x,r) \coloneqq \{ B_{j}^{D} = B_{U}(x_{j},r_{j}) \in \mathbb{S}(D) \mid D \in \mathfrak{R}(x,r) \}$.
    Hence
    \begin{align*}\label{e:PI.chain}
        &\sum_{i \in I; B_{U}(x_i,r_i) \in \mathfrak{R}(x,r)}\int_{B_{U}(x_{i},3r_{i})}\abs{u_{B_{U}(x_{i},3r_{i})} - u_{3B_{0}}}^{p}\,dm \nonumber \\
        &\overset{\eqref{e:VD.ud},\eqref{e:telescope.basis}}{\le} c_{2}\int_{X}\sum_{i \in I; D \coloneqq B_{U}(x_i,r_i) \in \mathfrak{R}(x,r)}\left(\sum_{B_{U}(x_j,r_j) \in \mathfrak{R}_{\mathrm{chain}}(x,r)}a_{j}\indicator{D}\indicator{B_{U}(x_{j},C_{1}r_{j})}\right)^{p}\,dm \nonumber \\
        &\le c_{3}\int_{X}\left(\sum_{B_{U}(x_j,r_j) \in \mathfrak{R}_{\mathrm{chain}}(x,r)}a_{j}\indicator{B_{U}(x_{j},r_{j})}\right)^{p}\,dm \qquad \text{(by Lemma \ref{lem:Lp.overlap})} \nonumber \\
        &\overset{\eqref{e:chainball.rad},\eqref{e:scalefunction}}{\le} c_{4}\Psi(r)\sum_{B_{U}(x_j,r_j) \in \mathfrak{R}_{\mathrm{chain}}(x,r)}\int_{B_{U}(x_{j},A_{\mathrm{P}}Lr_{j})}\,d\Gamma_{p}\langle u \rangle
        \overset{\eqref{e:whitney.finiteoverlap},\eqref{e:chainball.rad}}{\le} c_{5}\Psi(r)\int_{B_{U}(x,\lambda r)}\,d\Gamma_{p}\langle u \rangle.
    \end{align*}
    Here $\lambda$ is chosen so that $\lambda \ge C_{0} + A_{\mathrm{P}}L\frac{(A(4\varepsilon + 1) + 1 - 2\varepsilon)}{(1 - \varepsilon)^{2}}$, where $C_{0}$ is the constant in \eqref{e:chainball.center}.
\end{proof}

Now we prove the Poincar\'{e} inequality on uniform domains, which is an analogue of \cite[Theorem 5.3]{Mur24}. 
\begin{thm}[\ref{e:PI} on uniform domains]\label{thm:PI.ud}
    Assume that \ref{e:VD} and \ref{e:PI} hold.
    Let $U$ be a $A$-uniform domain for some $A \ge 1$.
    Then there exist $A_{U}, C_{U} \ge 1$ (depending only on $p,c_{\mathrm{D}},A,\varepsilon,C_{\Psi},\beta_{2},C_{\mathrm{P}},A_{\mathrm{P}}$) such that for any $x \in U$, $r > 0$, and $u \in \mathcal{F}_{p,\mathrm{loc}}(B_{U}(x,A_{U}r))$,
    \begin{equation}\label{e:PI.ud}
        \int_{B_{U}(x,r)}\abs{u - u_{B_{U}(x,r)}}^{p}\,dm \le C_{U}\Psi(r)\int_{B_{U}(x,A_{U}r)}\,d\Gamma_{p}\langle u \rangle.
    \end{equation}
\end{thm}
\begin{proof}
    As in \cite[Theorem 5.3]{Mur24}, the proof is divided into the following three cases.
    \begin{enumerate}[label=\textup{(\roman*)},align=left,leftmargin=*,topsep=2pt,parsep=0pt,itemsep=2pt]
        \item $x \in \closure{U}^{X}$ and $r \ge \delta_{U}(x)$.
        \item $x \in \closure{U}^{X}$ and $A_{\mathrm{P}}r \le \delta_{U}(x)$.
        \item $x \in U$ and $A_{\mathrm{P}}^{-1}\delta_{U}(x) < r < \delta_{U}(x)$.
    \end{enumerate}

    (i)\, Let $\varepsilon \in (0,1/4)$ be small enough so that $27A_{\mathrm{P}}\frac{\varepsilon}{1 + \varepsilon} < 1$ and let $\mathfrak{R} = \{ B_{U}(x_i,r_i) \}_{i \in I}$ be an $\varepsilon$-Whitney cover of $U$.
    Let $B_{0} = B_{U}(x(B_{0}),r(B_{0})) \in \mathfrak{R}(x,r)$ satisfy $r(B_{0}) \ge \frac{\varepsilon}{3A(4+\varepsilon)}r$, which exists by Proposition \ref{prop:whitney}-\ref{it:whitney.central}, and set $3B_{0} \coloneqq B_{U}(x(B_{0}),3r(B_{0}))$.
    By \cite[Lemma 4.17]{BB} and \eqref{e:whitney.nearball}, we observe that
    \begin{align}\label{e:prePIonUD.1}
        &\int_{B_{U}(x,r)}\abs{u - u_{B_{U}(x,r)}}^{p}\,dm
        \le 2^{p}\int_{B_{U}(x,r)}\abs{u - u_{3B_{0}}}^{p}\,dm \nonumber \\
        &\quad\le 2^{2p - 1}\sum_{i \in I; B_{U}(x_i,r_i) \in \mathfrak{R}(x,r)}\int_{B_{U}(x_{i},3r_{i})}\bigl\{\abs{u - u_{B_{U}(x_{i},3r_{i})}}^{p} + \abs{u_{B_{U}(x_{i},3r_{i})} - u_{3B_{0}}}^{p}\bigr\}\,dm.
    \end{align}
    Let $A_{U} > 0$ satisfy $A_{U} \ge \bigl(1 + 6\frac{1 + 4\varepsilon}{1 - 2\varepsilon}\bigr) \vee \lambda$, where $\lambda \in [1,\infty)$ is the constant in Lemma \ref{lem:telescope}. 
    Then we have from \ref{e:PI}, \eqref{e:scalefunction}, \eqref{e:rad.comparison1} and \eqref{e:whitney.finiteoverlap} that  
    \begin{equation}\label{e:prePIonUD.2}
        \sum_{i \in I; B_{U}(x_i,r_i) \in \mathfrak{R}(x,r)}\int_{B_{U}(x_{i},3r_{i})}\abs{u - u_{B_{U}(x_{i},3r_{i})}}^{p}\,dm
        \le c_{1}\Psi(r)\int_{B_{U}(x,A_{U}r)}\,d\Gamma_{p}\langle u \rangle. 
    \end{equation}
    From Lemma \ref{lem:telescope}, we also have
    \begin{equation}\label{e:prePIonUD.3}
        \sum_{i \in I; B_{U}(x_i,r_i) \in \mathfrak{R}(x,r)}\int_{B_{U}(x_{i},3r_{i})}\abs{u_{B_{U}(x_{i},3r_{i})} - u_{3B_{0}}}^{p}\,dm
        \le c_{2}\Psi(r)\int_{B_{U}(x,A_{U}r)}\,d\Gamma_{p}\langle u \rangle. 
    \end{equation}
    We then obtain \eqref{e:PI.ud} by combining \eqref{e:prePIonUD.1}, \eqref{e:prePIonUD.2} and \eqref{e:prePIonUD.3}. 

    (ii)\, In this case, by \ref{e:PI}, we obtain \eqref{e:PI.ud} with $C_{U} = C_{\mathrm{P}}$ and $A_{U} = A_{\mathrm{P}}$.

    (iii)\, This case can be shown by considering $B_{U}(x,A_{\mathrm{P}}r)$ instead of $B_{U}(x,r)$ in (i).
\end{proof}

\subsection{Extension map and its scale-invariant boundedness}
In the following proposition, we recall from \cite{Mur24} a reflection map $Q$, which is motivated by Jones \cite{Jon81}.
\begin{prop}[Reflection map; {\cite[Proposition 3.12]{Mur24}}]\label{prop:ref}
    Assume that $(X,d)$ is a doubling metric space.
    Let $U$ be a $A$-uniform domain for some $A \ge 1$ with $U^{\#} \coloneqq \mathrm{int}(X \setminus U) \neq \emptyset$, let $\varepsilon \in (0,1/14)$ and let $\mathfrak{R}$ be an $\varepsilon$-Whitney cover of $U$.
    Let $\mathfrak{R}^{\#}$ be an $\varepsilon$-Whitney cover of $U^{\#}$ and define
    \[
    \widetilde{\mathfrak{R}}^{\#}
    \coloneqq \biggl\{ B_{U^{\#}}(y,s) \in \mathfrak{R}^{\#} \biggm| s < \frac{\varepsilon}{6A(1 + \varepsilon)}\diam(U) \biggr\}.
    \]
    Then there exists a map $Q \colon \widetilde{\mathfrak{R}}^{\#} \to \mathfrak{R}$ such the the following properties hold:
    \begin{enumerate}[label=\textup{(\alph*)},align=left,leftmargin=*,topsep=2pt,parsep=0pt,itemsep=2pt]
        \item For any $B = B_{U^{\#}}(y,s) \in \widetilde{\mathfrak{R}}^{\#}$, the ball $Q(B) = B_{U}(x,r) \in \mathfrak{R}$ satisfies
        \begin{equation}\label{e:ref.rad}
            \frac{1 + \varepsilon}{1 + 4\varepsilon}s < r < \frac{1 + \varepsilon}{1 - 2\varepsilon}s, \quad d(x,y) \le \left(2 + \frac{3A}{2}\right)\frac{1 + \varepsilon}{\varepsilon}s.
        \end{equation}
        \item There exists $K$ depending only on $\varepsilon,A,N_{\mathrm{D}}$ such that $\sup\{ \#Q^{-1}(B) \mid B \in \mathfrak{R} \} \le K$.
        \item Let $B_{U^{\#}}(y_{i},s_{i}) \in \widetilde{\mathfrak{R}}^{\#}$ and $B_{U}(x_{i},r_{i}) = Q(B_{U^{\#}}(y_{i},s_{i})) \in \mathfrak{R}$, $i = 1,2$. If $B_{U^{\#}}(y_{1},6s_{1}) \cap B_{U^{\#}}(y_{2},6s_{2}) \neq \emptyset$, then
        \begin{equation}\label{e:ref.preball.rad}
            \frac{(1 - 2\varepsilon)(1 - 5\varepsilon)}{(1 + 4\varepsilon)(1 + 7\varepsilon)}r_{2} \le r_{1} \le \frac{(1 + 4\varepsilon)(1 + 7\varepsilon)}{(1 - 2\varepsilon)(1 - 5\varepsilon)}r_{2},
        \end{equation}
        and there is a chain of distinct balls $\{ B_{U}(z_{j}, t_{j}) \in \mathfrak{R} \}_{j = 1}^{N}$ such that $z_{1} = x_{1}$, $t_{1} = r_{1}$, $z_{N} = x_{2}$, $t_{N} = r_{2}$, $B_{U}(z_{j},3t_{j}) \cap B_{U}(z_{j + 1},3t_{j + 1}) \neq \emptyset$ for any $j \in \{ 1,\dots,N-1 \}$ where $N$ satisfies $N \le N_{0}$ for some $N_{0} \in \mathbb{N}$ depending only on $\varepsilon,A,N_{\mathrm{D}}$.
    \end{enumerate}
\end{prop}

In the rest of this section, we assume the same conditions as in Proposition \ref{prop:ref}, i.e., $U$ is a $A$-uniform domain for some $A \ge 1$ with $U^{\#} \coloneqq \mathrm{int}(X \setminus U) \neq \emptyset$, $\varepsilon \in (0,1/14)$ and $\mathfrak{R}, \mathfrak{R}^{\#}$ are $\varepsilon$-Whitney covers of $U, U^{\#}$ respectively.
In addition, we assume that \ref{e:VD} and \ref{e:capu} hold.
Let $\{ \psi_{B} \}_{B \in \mathfrak{R}^{\#}}$ be a controlled partition of unity given by Lemma \ref{lem:unity} with $A = 2$, which exists by Remark \ref{rmk:whitney-good} and Proposition \ref{prop:capu.refine}.
As a result, we have $\supp_{X}[\psi_{B}] \subset 6B \coloneqq B(y,6s)$ for any $B = B_{U^{\#}}(y,s) \in \mathfrak{R}^{\#}$.
For $B \in \widetilde{\mathfrak{R}}^{\#}$ and $\lambda \ge 1$, we write $\lambda Q(B) = B_{U}(x,\lambda r)$ where $B_{U}(x,r) = Q(B) \in \mathfrak{R}$.
We define
\[
\mathfrak{R}^{\#}(\xi,r)
\coloneqq \{ B_{U^{\#}}(y,s) \in \mathfrak{R}^{\#} \mid B_{U^{\#}}(y,6s) \cap B(\xi,r) \neq \emptyset \},
\]
and $\widetilde{\mathfrak{R}}^{\#}(\xi,r) \coloneqq \mathfrak{R}^{\#}(\xi,r) \cap \widetilde{\mathfrak{R}}^{\#}$.

The following proposition collects basic propeties that are frequently used in the proof of Proposition \ref{prop:extQ} later.
See \cite[Proposition 3.12 and Lemma 3.14]{Mur24} for details.
\begin{prop}\label{prop:refcovering}
    \begin{enumerate}[label=\textup{(\alph*)},align=left,leftmargin=*,topsep=2pt,parsep=0pt,itemsep=2pt]
        \item\label{it:ref.radiicomparison} For any $B = B_{U^{\#}}(y,s) \in \widetilde{\mathfrak{R}}^{\#}$, the ball $Q(B) = B_{U}(x,r) \in \mathfrak{R}$ satisfies
        \[
        \frac{1 + \varepsilon}{1 + 4\varepsilon}s < r < \frac{1 + \varepsilon}{1 - 2\varepsilon}s,
        \quad d(x,y) \le \left(2 + \frac{3A}{2}\right)\frac{1 + \varepsilon}{\varepsilon}s.
        \]
        In particular, for any $\lambda \ge 1$, there exist $C, K \ge 1$ depending only on $\varepsilon,A,c_{\mathrm{D}},\lambda$ such that
        \[
        C^{-1}m(B) \le m(\lambda Q(B)) \le Cm(B),
        \quad \text{for any $B \in \widetilde{\mathfrak{R}}^{\#}$}
        \]
        and for any $\xi \in \partial U$, $r > 0$,
        \[
        \bigl\{ \lambda Q(B) \bigm| B \in \widetilde{\mathfrak{R}}^{\#}(\xi,r) \bigr\} \subset \mathfrak{R}(\xi,Kr).
        \]
        \item\label{it:ref.Qoverlap} There exists $N \in \mathbb{N}$ depending only on $\varepsilon,A,N_{\mathrm{D}}$ such that $\sup_{B \in \mathfrak{R}}\#(Q^{-1}(B)) \le N$.
        \item\label{it:ref.chain} Let $B_{U^{\#}}(y_{i},s_{i}) \in \widetilde{\mathfrak{R}}^{\#}$, $i = 1,2$ satisfy $B_{U^{\#}}(y_{1},6s_{1}) \cap B_{U^{\#}}(y_{2},6s_{2}) \neq \emptyset$ and let $B_{U}(x_{i},r_{i}) = Q(B_{U^{\#}}(y_{i},s_{i}))$.
        Then there exist $L,K \ge 1$ depending only on $\varepsilon,A,N_{\mathrm{D}}$ such that the following conditions hold.
        There is a family of distinct balls $\{ B_{j} \coloneqq B_{U}(z_{j},t_{j}) \in \mathfrak{R} \}_{j = 1}^{L_{0}}$ with $L_{0} \le L$ such that $B_{1} = B_{U}(x_{1},r_{1})$, $B_{L_{0}} = B_{U}(x_{2},r_{2})$ and $3B_{j} \cap 3B_{j + 1} \neq \emptyset$ for $j \in \{ 1,\dots,L_{0}-1 \}$. Furthermore, if $B_{1} \in \mathfrak{R}(\xi, r)$ for some $\xi \in \partial U$ and $r > 0$, then $\{ B_{j} \}_{j = 1}^{L_{0}} \subset \mathfrak{R}(\xi,Kr)$.
    \end{enumerate}
\end{prop}

Now we define the extension map $\Ext_{Q} \colon L^{p}(U,m|_{U}) \to L^{p}(X,m)$ by
\begin{equation}\label{e:defn.extmap}
    \Ext_{Q}(u)(x) \coloneqq
    \begin{cases}
        u(x) \quad &\text{if $x \in U$,} \\
        \sum_{B \in \widetilde{\mathfrak{R}}^{\#}}u_{3Q(B)}\,\psi_{B}(x) \quad &\text{otherwise;}
    \end{cases}
\end{equation}
the fact that $\Ext_{Q}(u) \in L^{p}(X,m)$ is shown in the following lemma. (Recall that $m(\partial U) = 0$ by Proposition \ref{prop:geom.ud}-\ref{it:VD.ud}.)
\begin{lem}
    Let $\Ext_{Q}$ be the extension map as defined in \eqref{e:defn.extmap}.
    Then $\Ext_{Q} \colon L^{p}(U,m|_{U}) \to L^{p}(X,m)$ is a bounded linear operator, and there exist $C_{1}, A_{1} \in [1,\infty)$ (depending only on $p,\varepsilon,A,N_{\mathrm{D}}$) such that for any $u \in L^{p}(U, m|_{U})$, $\xi \in \partial U$ and $r > 0$,
    \begin{equation}\label{e:ref.Lp}
        \int_{B(\xi,r)}\abs{\Ext_{Q}(u)}^{p}\,dm \le C_{1}\int_{B_{U}(\xi,A_{1}r)}\abs{u}^{p}\,dm.
    \end{equation}
\end{lem}
\begin{proof}
    The linearity of $\Ext_{Q}$ is clear and the boundedness of $\Ext_{Q}$ follows from \eqref{e:ref.Lp}, so we will prove \eqref{e:ref.Lp}.
    Since $\supp_{X}[\psi_{B}] \subset 6B$, we have
    \[
    (\Ext_{Q}(u))\big|_{B(\xi,r) \setminus U} = \sum_{B \in \widetilde{\mathfrak{R}}^{\#}(\xi,r)}u_{3Q(B)}\psi_{B}\big|_{B(\xi,r) \setminus U}
    \]
    Note that, by Propositions \ref{prop:ref} and \ref{prop:whitney}-\ref{it:whitney.overlap}, there exist $c_{1},A_{1} \in [1,\infty)$ depending only on $\varepsilon,A,N_{\mathrm{D}}$ such that
    \begin{equation}\label{e:ref.Lp.overlap}
        \sum_{B \in \widetilde{\mathfrak{R}}^{\#}(\xi,r)}\indicator{3Q(B)} \le c_{1}\indicator{B_{U}(\xi,A_{1}r)}.
    \end{equation}
    Since $\partial U^{\#} \subset \partial U$ by \cite[Lemma 3.11]{Mur24}, we have from Proposition \ref{prop:geom.ud}-\ref{it:VD.ud} that $m(\partial U^{\#}) = 0$, which together with $\psi_{B} \le \indicator{6B}$ implies
    \begin{equation*}
        \int_{B(\xi,r)}\abs{\Ext_{Q}(u)}^{p}\,dm
        \le \int_{B(\xi,r) \cap U}\abs{u}^{p}\,dm + \int_{U^{\#}}\abs{\sum_{B \in \widetilde{\mathfrak{R}}^{\#}(\xi,r)}u_{3Q(B)}\indicator{6B}}^{p}\,dm.
    \end{equation*}
    To estimate the second term, we note that $\sum_{B \in \widetilde{\mathfrak{R}}^{\#}(\xi,r)}\indicator{6B} \le c_{2}$ by Proposition \ref{prop:whitney}-\ref{it:whitney.overlap} and $\varepsilon < 1/6$.
    By H\"{o}lder's inequality, \eqref{e:ref.rad}, \ref{e:VD} and \eqref{e:ref.Lp.overlap},
    \begin{align*}
        \int_{U^{\#}}\abs{\sum_{B \in \widetilde{\mathfrak{R}}^{\#}(\xi,r)}u_{3Q(B)}\indicator{6B}}^{p}\,dm
        \le c_{3}\sum_{B \in \widetilde{\mathfrak{R}}^{\#}(\xi,r)}\int_{3Q(B)}\abs{u}^{p}\,dm
        \le c_{1}c_{3}\int_{B_{U}(\xi,A_{1}r)}\abs{u}^{p}\,dm,
    \end{align*}
    completing the proof.
\end{proof}

The following proposition provides estimates on the $p$-energy measure of $\Ext_{Q}(u)$, which is an analogue of \cite[Proposition 5,8]{Mur24}.
\begin{prop}\label{prop:extQ}
    Assume that \ref{e:VD}, \ref{e:PI} and \ref{e:capu} hold.
    Let $U$ be a $A$-uniform domain for some $A \ge 1$ with $U^{\#} \coloneqq \mathrm{int}(X \setminus U) \neq \emptyset$, let $\varepsilon \in (0,1/14)$ and let $\mathfrak{R}, \mathfrak{R}^{\#}$ be $\varepsilon$-Whitney covers of $U, U^{\#}$ respectively.
    Let $\Ext_{Q}$ be the extension map as defined in \eqref{e:defn.extmap}.
    Then there exist $C,K \in [1,\infty)$ (depending only on $p,A,\varepsilon$ and the constants associated with \ref{e:VD}, \eqref{e:scalefunction}, \ref{e:PI} and \ref{e:capu}) such that the following conditions hold.
    \begin{enumerate}[label=\textup{(\alph*)},align=left,leftmargin=*,topsep=2pt,parsep=0pt,itemsep=2pt]
        \item\label{it:ref.energy.a} There exists $a_{0} \in (0,1)$ such that for any $\xi \in \partial U$, $0 < r < a_{0}\diam(U)$ and $u \in \mathcal{F}_{p}(U)$,
        \begin{equation}\label{e:ext.PI}
            \inf_{\alpha \in \mathbb{R}}\int_{B(\xi,r)}\abs{\Ext_{Q}(u) - \alpha}^{p}\,dm \le C\Psi(r)\int_{B_{U}(\xi,Kr)}\,d\Gamma_{p,U}\langle u \rangle.
        \end{equation}
        \item\label{it:ref.energy.b} For any $u \in L^{p}(U,m|_{U})$, we have $\Ext_{Q}(u) \in \mathcal{F}_{p,\mathrm{loc}}(U^{\#})$.
        Moreover, there exists $a_{1} \in (0,1)$ such that for any $u \in \mathcal{F}_{p}(U)$, $\xi \in \partial U$ and $0 < r < a_{1}\diam(U)$, we have
        \begin{equation}\label{e:extEM.local}
            \Gamma_{p}\langle \Ext_{Q}(u) \rangle(B_{U^{\#}}(\xi,r)) \le C\Gamma_{p,U}\langle u \rangle(B_{U}(\xi,Kr)),
        \end{equation}
        and
        \begin{equation}\label{e:extEM.total}
            \Gamma_{p}\langle \Ext_{Q}(u) \rangle(U^{\#}) \le C\left(\Gamma_{p,U}\langle u \rangle(U) + \frac{1}{\Psi(\diam(U))}\int_{U}\abs{u}^{p}\,dm\right),
        \end{equation}
        where by convention that $\frac{1}{\Psi(\diam(U))} = 0$ if $\diam(U) = \infty$.
        \item\label{it:ref.energy.c} For any $u \in \mathcal{F}_{p}(U)$, we have $\Ext_{Q}(u) \in \mathcal{F}_{p}$ and
        \begin{equation}\label{e:extEp}
            \mathcal{E}_{p}(\Ext_{Q}(u)) \le C\left(\Gamma_{p,U}\langle u \rangle(U) + \frac{1}{\Psi(\diam(U))}\int_{U}\abs{u}^{p}\,dm\right).
        \end{equation}
        \item\label{it:ref.energy.d} For any $u \in \mathcal{F}_{p}(U)$, we have
        \begin{equation}\label{e:extQ.bdry}
            \Gamma_{p}\langle \Ext_{Q}(u) \rangle(\partial U) = 0.
        \end{equation}
        \item\label{it:ref.energy.e} There exists $a_{2} \in (0,1)$ such that for any $u \in \mathcal{F}_{p}(U)$, $x \in \closure{U}^{X}$ and $0 < r < a_{2}\diam(U)$,
        \begin{equation}
            \Gamma_{p}\langle \Ext_{Q}(u) \rangle(B(x,r)) \le C\Gamma_{p,U}\langle u \rangle(B_{U}(x,Kr)).
        \end{equation}
        \item\label{it:ref.energy.f} $\Gamma_{p}\langle u \rangle(\partial U) = 0$ for any $u \in \mathcal{F}_{p}$.
    \end{enumerate}
\end{prop}
\begin{proof} 
    \ref{it:ref.energy.a}:
    Let $B_0 = B(x_0,r_0) \in \mathfrak{R}(\xi,r)$ satisfy \eqref{e:centralball.rad}.
    By the argument in the case (i) in the proof of Theorem \ref{thm:PI.ud}, we have
    \[
    \int_{B_{U}(\xi,r)}\abs{\Ext_{Q}(u) - u_{B_{U}(x_{0},3r_{0})}}^{p}\,dm
    \le c_{1}\Psi(r)\int_{B_{U}(\xi,A_{U}r)}\,d\Gamma_{p}\langle u \rangle.
    \]
    By \eqref{e:ref.rad}, there exists $A_{1} \ge 1$ depending only on $\varepsilon,A$ such that $\{ Q(B) \mid B \in \widetilde{\mathfrak{R}}^{\#}(\xi,r) \} \subset \mathfrak{R}(\xi,A_{1}r)$.
    Now, by H\"{o}lder's inequality, \eqref{e:whitney.finiteoverlap} $\psi_{B} \le \indicator{6B}$, \ref{e:VD} and Lemma \ref{lem:telescope}, 
    \begin{align*}
        &\int_{B_{U^{\#}}(\xi,r)}\abs{\Ext_{Q}(u) - u_{B_{U}(x_{0},3r_{0})}}^{p}\,dm 
        = \int_{B_{U^{\#}}(\xi,r)}\abs{\sum_{B \in \widetilde{\mathfrak{R}}^{\#}(\xi,r)}\bigl(u_{3Q(B)} - u_{B_{U}(x_{0},3r_{0})}\bigr)\psi_{B}}^{p}\,dm \\
        &\qquad\le c_{2}\sum_{B \in \mathfrak{R}(\xi,A_{1}r)}\abs{u_{3B} - u_{B_{U}(x_{0},3r_{0})}}^{p}m(3B)
        \le c_{3}\Psi(r)\int_{B_{U}(\xi,\lambda A_{1}r)}\,d\Gamma_{p}\langle u \rangle.
    \end{align*}
    These estimates above imply \eqref{e:ext.PI} with $K \coloneqq A_{U} \vee (\lambda A_{1})$.

    \ref{it:ref.energy.b}:
    For any $B_{U^{\#}}(y_{0},s_{0}) \in \mathfrak{R}^{\#}$ and $u \in \mathcal{F}_{p}(U)$, by $\supp_{X}[\psi_{B}] \subset 6B$, we observe that for any fixed $\alpha \in \mathbb{R}$,
    \begin{equation}\label{e:ExtQ.localrep}
        \Ext_{Q}(u)\big|_{B_{U^{\ast}}(y_{0},3s_{0})}
        = \alpha + \sum_{B \in \widetilde{\mathfrak{R}}^{\#}(y_{0},3s_{0})}\bigl(u_{3Q(B)} - \alpha\bigr)\,\psi_{B}\big|_{B_{U^{\ast}}(y_{0},3s_{0})}.
    \end{equation}
    Noting that there exists $N_{1} \in \mathbb{N}$ depending only on $N_{\mathrm{D}},\varepsilon$ such that $\#\bigl(\widetilde{\mathfrak{R}}^{\#}(y_{0},3s_{0})\bigr) \le N_{1}$ by \eqref{e:whitney.finiteoverlap}, we have $\Ext_{Q}(u) \in \mathcal{F}_{p}(B_{U^{\#}}(y_{0},3s_{0}))$.
    Since $B_{U^{\#}}(y_{0},s_{0}) \in \mathfrak{R}^{\#}$ is arbitrary, we conclude $\Ext_{Q}(u) \in \mathcal{F}_{p,\mathrm{loc}}(U^{\#})$.
    We also see from \ref{it:EMtri}-\ref{it:EMslocal} of Framework \ref{frame:EPEM}, \eqref{e:ExtQ.localrep} and \eqref{e:low-energy} that
    \begin{equation}\label{e:extEM.pre}
        \Gamma_{p}\langle \Ext_{Q}(u) \rangle(B_{U^{\ast}}(y_{0},3s_{0}))
        \le c_{4}\inf_{\alpha \in \mathbb{R}}\sum_{B \in \widetilde{\mathfrak{R}}^{\#}(y_{0},3s_{0})}\abs{u_{3Q(B)} - \alpha}^{p}\frac{m(Q(B))}{\Psi(r(Q(B)))},
    \end{equation}
    where $r(Q(B))$ denotes the radius the ball $Q(B)$.
    As in \cite[(5.24)]{Mur24}, by using Proposition \ref{prop:refcovering}, we find $a_{1} \in (0,1)$ such that if $\xi \in \partial U$, $0 < r < a_{1}\diam(U)$ and $B_{i} \in \mathfrak{R}^{\#}$, $i \in \{ 1,2 \}$, satisfy $\{ B_{1}, B_{2} \} \cap \mathfrak{R}^{\#}(\xi,r) \neq \emptyset$ and $6B_{1} \cap B_{2} \neq \emptyset$, then $B_{1},B_{2} \in \widetilde{\mathfrak{R}}^{\#}$. 
    In particular, for such $\xi$ and $r$, we have $B \in \widetilde{\mathfrak{R}}^{\#}(\xi,r)$ whenever $B \in \mathfrak{R}^{\#}$ satisfies $3B \cap B_{U^{\#}}(\xi,r) \neq \emptyset$, thereby that 
    \begin{align}\label{e:ExtQ.cover}
    	\Gamma_{p}\langle \Ext_{Q}(u) \rangle & (B_{U^{\#}}(\xi,r))
        \le \sum_{B_{U^{\#}}(y_0,s_0) \in \widetilde{\mathfrak{R}}^{\#}(\xi,r)}\Gamma_{p}\langle \Ext_{Q}(u) \rangle(B_{U^{\ast}}(y_{0},3s_{0})) \nonumber \\
        &\overset{\eqref{e:extEM.pre}}{\le} c_{4}\sum_{\substack{B' = B_{U^{\#}}(y_,s_0) \in \widetilde{\mathfrak{R}}^{\#}(\xi,r), \\ B \in \widetilde{\mathfrak{R}}^{\#}(y_{0},3s_{0})}}\abs{u_{3Q(B)} - u_{3Q(B')}}^{p}\frac{m(Q(B))}{\Psi(r(Q(B)))} \nonumber \\
        &\;\le c_{5}\sum_{\substack{B_{1},B_{2} \in \mathfrak{R}(\xi,K_{1}r) \\ 3B_{1} \cap 3B_{2} \neq \emptyset}}\abs{u_{3B_{1}} - u_{3B_{2}}}^{p}\frac{m(B_{1})}{\Psi(r(B_{1}))}, 
    \end{align}
    where, in the last line, we used \cite[(5.26) and (5.27)]{Mur24} and the stability of discrete $p$-energy forms under rough isometries (e.g., \cite[Lemma A.1]{Shi24}) along with \cite[Lemma 3.14-(b)]{Mur24}, and we choosed the constant $K_{1} \ge 1$ (independent of $\xi$, $r$ and $u$) as in \cite[(5.26)]{Mur24}. 
    By Lemma \ref{lem:telescope} and Proposition \ref{prop:whitney}-\ref{it:whitney.radcomparison},\ref{it:whitney.overlap}, there exists $K \ge 1$ (independent of $u$, $\xi$ and $r$) such that 
    \begin{equation}\label{e:ExtQ.cover.PI}
    	\sum_{\substack{B_{1},B_{2} \in \mathfrak{R}(\xi,K_{1}r) \\ 3B_{1} \cap 3B_{2} \neq \emptyset}}\abs{u_{3B_{1}} - u_{3B_{2}}}^{p}\frac{m(B_{1})}{\Psi(r(B_{1}))} 
    	\le c_{6}\,\Gamma_{p}\langle u \rangle(B_{U}(\xi,Kr)). 
    \end{equation}
    We obtain \eqref{e:extEM.local} by combining \eqref{e:ExtQ.cover} and \eqref{e:ExtQ.cover.PI}.  
    One can also obtain \eqref{e:extEM.total} in the same manner as in \cite[Proof of Proposition 5.8-(b)]{Mur24} (by using \eqref{e:extEM.pre} instead of \cite[(5.23)]{Mur24} in \cite[(5.33)]{Mur24}). 

    \ref{it:ref.energy.c}:
    Let $r > 0$, let $\mathcal{N}_{1} \subset \partial U$ be a $r$-net of $\partial U$, let $\mathcal{N}$ be a $r$-net of $X$ with $\mathcal{N}_{1} \subset \mathcal{N}$ and define $\mathcal{N}_{2} \coloneqq \{ z \in \mathcal{N} \mid B(z,2A_{\mathrm{P}}r) \cap \partial U = \emptyset \}$.
    Then we have
    \begin{equation}\label{e:ext.cover}
        X = \left(\bigcup_{z \in \mathcal{N}_{1}}B(z,2(A_{\mathrm{P}} + 1)r)\right) \cup \left(\bigcup_{z \in \mathcal{N}_{2}}B(z,r)\right);  
    \end{equation}
    see  \cite[(5.35)]{Mur24}.
    By \eqref{e:ext.cover} and \cite[(5.37)]{Mur24}, we can see that 
    \begin{align}\label{e:ext.KSdivide}
        W_{p}(\Ext_{Q}(u),r)
        &\coloneqq  \frac{1}{\Psi(r)}\int_{X}\fint_{B(y,r)}\abs{\Ext_{Q}(u)(x) - \Ext_{Q}(u)(y)}^{p}\,m(dx)\,m(dy) \nonumber \\
        &\le \frac{c_{7}}{\Psi(r)}\Biggl(\sum_{z \in \mathcal{N}_{1}}\int_{B(z,2(A_{\mathrm{P}} + 3)r)}\abs{\Ext_{Q}(u)(x)- (\Ext_{Q}(u))_{B(z,2(A_{\mathrm{P}}+3)r)}}^{p}\,m(dx) \nonumber \\
        &\qquad\qquad+ \sum_{z \in \mathcal{N}_{2}}\int_{B(z,2r)}\abs{\Ext_{Q}(u)(x)- (\Ext_{Q}(u))_{B(z,2r)}}^{p}\,m(dx)\Biggr).
    \end{align}
    Similar to \cite[(5.39)]{Mur24}, by $\partial U^{\#} \subset \partial U$, \ref{e:PI}, \eqref{e:scalefunction}, \ref{e:VD} and \eqref{e:extEM.total},
    \begin{align*}
        &\frac{1}{\Psi(r)}\sum_{z \in \mathcal{N}_{2}}\int_{B(z,2r)}\abs{\Ext_{Q}(u)(x)- (\Ext_{Q}(u))_{B(z,2r)}}^{p}\,m(dx) \\
        &\le c_{8}\Bigl(\Gamma_{p,U}\langle u \rangle(U) + \Gamma_{p}\langle \Ext_{Q}(u) \rangle(U^{\#})\Bigr)
        \le c_{9}\left(\Gamma_{p,U}\langle u \rangle(U) + \frac{1}{\diam(U)}\int_{U}\abs{u}^{p}\,dm\right).
    \end{align*}
    For any small enough $r > 0$, by \eqref{e:ext.PI}, \eqref{e:scalefunction} and \ref{e:VD},
    \[
    \frac{1}{\Psi(r)}\sum_{z \in \mathcal{N}_{1}}\int_{B(z,2(A_{\mathrm{P}} + 3)r)}\abs{\Ext_{Q}(u)(x)- (\Ext_{Q}(u))_{B(z,2(A_{\mathrm{P}}+3)r)}}^{p}\,m(dx)
    \le c_{10}\Gamma_{p,U}\langle u \rangle(U).
    \]
    These estimates along with \eqref{e:ext.KSdivide} imply that
    \begin{equation*}
        \limsup_{r \downarrow 0}W_{p}(\Ext_{Q}(u),r)
        \le c_{11}\left(\Gamma_{p,U}\langle u \rangle(U) + \frac{1}{\diam(U)}\int_{U}\abs{u}^{p}\,dm\right),
    \end{equation*}
    which together with Corollaries \ref{cor:main1.2} and \ref{cor:main1.3} yields $\Ext_{Q}(u) \in \mathcal{F}_{p}$ and \eqref{e:extEp}.

    \ref{it:ref.energy.d}:
    Let $\mathcal{N}_{1}, \mathcal{N}, \mathcal{N}_{2}$ be chosen as in the proof of \ref{it:ref.energy.c}.
    Then, by using \cite[(5.42)]{Mur24},
    we see that for any $\delta > 0$, $r > 0$ and $u \in \mathcal{F}_{p}(U)$,
    \begin{align}\label{e:extKS.local}
        &\frac{1}{\Psi(r)}\int_{(\partial U)(\delta)}\fint_{B(y,r)}\abs{\Ext_{Q}(u)(x) - \Ext_{Q}(u)(y)}^{p}\,m(dx)\,m(dy) \nonumber \\
        &\qquad \le \frac{c_{12}}{\Psi(r)}\biggl(\sum_{z \in \mathcal{N}_{1}}\int_{B(z,(2A_{\mathrm{P}}+3)r)}\abs{\Ext_{Q}(u)(x) - (\Ext_{Q}(u))_{B(z,(2A_{\mathrm{P}}+3)r)}}^{p}\,m(dx) \nonumber \\
        &\qquad\qquad\qquad + \sum_{z \in \mathcal{N}_{2} \cap (\partial U)(\delta + r)}\int_{B(z,2r)}\abs{\Ext_{Q}(u)(x) - (\Ext_{Q}(u))_{B(z,2r)}}^{p}\,m(dx)\biggr).
    \end{align}
    Similar to \cite[(5.45)-(5.47)]{Mur24}, for any small enough $\delta,r > 0$ with
    \[
    r < \frac{a_{0}}{2A_{\mathrm{P}} + 3}\diam(U) \quad \text{and} \quad \delta + 2(A_{\mathrm{P}}+1)r < a_{1}\diam(U),
    \]
    we have from \ref{e:PI}, \ref{e:VD}, \eqref{e:ext.PI} and \eqref{e:extEM.local} that
    \begin{align}\label{e:extKS.local.1}
        &\frac{1}{\Psi(r)}\sum_{z \in \mathcal{N}_{2} \cap (\partial U)(\delta + r)}\int_{B(z,2r)}\abs{\Ext_{Q}(u)(x) - (\Ext_{Q}(u))_{B(z,2r)}}^{p}\,m(dx) \nonumber \\
        &\qquad\qquad \le c_{13}\Gamma_{p,U}\langle u \rangle(U \cap (\partial U)(K(\delta + (2A_{\mathrm{P}}+3)r))),
    \end{align}
    and
    \begin{align}\label{e:extKS.local.2}
        &\frac{1}{\Psi(r)}\sum_{z \in \mathcal{N}_{1}}\int_{B(z,(2A_{\mathrm{P}}+3)r)}\abs{\Ext_{Q}(u)(x) - (\Ext_{Q}(u))_{B(z,(2A_{\mathrm{P}}+3)r)}}^{p}\,m(dx) \nonumber \\
        &\qquad\qquad \le c_{14}\Gamma_{p,U}\langle u \rangle(U \cap (\partial U)(K(2A_{\mathrm{P}}+3)r)),  
    \end{align}
    where $K \in [1,\infty)$ is the constant in Proposition \ref{prop:extQ}-\ref{it:ref.energy.a},\ref{it:ref.energy.b}.
    By \eqref{e:extKS.local}, \eqref{e:extKS.local.1} and \eqref{e:extKS.local.2}, we conclude that
    \[
    \lim_{\delta \downarrow 0}\limsup_{r \downarrow 0}\frac{1}{\Psi(r)}\int_{(\partial U)(\delta)}\fint_{B(y,r)}\abs{\Ext_{Q}(u)(x) - \Ext_{Q}(u)(y)}^{p}\,m(dx)\,m(dy) = 0,
    \]
    which along with Theorem \ref{thm:lower} (with $(\partial U)(\delta)$ in place of $\Omega$) yields $\Gamma_{p}\langle \Ext_{Q}(u) \rangle(\partial U) = 0$.

    \ref{it:ref.energy.e}:
    Thanks to the property in Framework \ref{frame:EPEM}-\ref{it:EMslocal}, the same argument as in \cite[Proof of Proposition 5.8-(e)]{Mur24} works.

    \ref{it:ref.energy.f}:
    It suffices to show that $\Gamma_{p}\langle u \rangle(\partial U) = 0$ for any $u \in \mathcal{F}_{p} \cap C_{c}(X)$ by virtue of the properties in Framework \ref{frame:EPEM}-\ref{it:regular},\ref{it:EMtri}.
    We define $\widehat{\Ext}_{Q}(u) \colon X \to \mathbb{R}$ by
    \begin{equation}\label{e:defn.extmap.cptsupp}
        \widehat{\Ext}_{Q}(u)(x) \coloneqq
        \begin{cases}
            u(x) \quad &\text{if $x \in \closure{U}^{X}$,} \\
            \sum_{B \in \widetilde{\mathfrak{R}}^{\#}}u_{3Q(B)}\psi_{B}(x) \quad &\text{otherwise;}
        \end{cases}
    \end{equation}
    Then one can see that $\widehat{\Ext}_{Q}(u) \in C_{c}(X)$ by following \cite[Proof of Lemma 5.9]{Mur24}.
    We note that $\widehat{\Ext}_{Q}(u) = \Ext_{Q}(u)$ $m$-a.e.\ since $m(\partial U) = 0$ by Proposition \ref{prop:geom.ud}-\ref{it:VD.ud}.
    From Framework \ref{frame:EPEM}-\ref{it:EMslocal} along with the fact that $u - \widehat{\Ext}_{Q}(u) = 0$ on $\closure{U}^{X}$ and $u - \widehat{\Ext}_{Q}(u) \in C_{c}(X)$, we have $\Gamma_{p}\langle u - \widehat{\Ext}_{Q}(u) \rangle(\closure{U}^{X}) = 0$. 
    This together with \eqref{e:extQ.bdry} and Framework \ref{frame:EPEM}-\ref{it:EMtri} implies that
    \begin{equation*}
        \Gamma_{p}\langle u \rangle(\partial U)
        = \abs{\Gamma_{p}\langle u \rangle(\partial U)^{1/p} - \Gamma_{p}\langle \Ext_{Q}(u) \rangle(\partial U)^{1/p}}^{p}
        \le \Gamma_{p}\langle u - \widehat{\Ext}_{Q}(u) \rangle(\partial U)
        = 0,
    \end{equation*}
    completing the proof.
\end{proof}

\begin{proof}[Proof of Theorem \ref{thm:ud.sext}]
    This follows from Proposition \ref{prop:extQ}-\ref{it:ref.energy.c},\ref{it:ref.energy.f}.
\end{proof}

\appendix
\section{Connectedness of the underlying metric space}\label{sec:conn}
In this section, under Framework \ref{frame:EPEM} and the assumption that any metric ball in $(X,d)$ is relatively compact, we show that the connectedness of $(X,d)$ is implied by \ref{e:PI} as in the case of PI-spaces \cite[Proposition 4.2]{BB}.  
Let $(X,d,m,\mathcal{E}_{p},\mathcal{F}_{p},\Gamma_{p}\langle \,\cdot\, \rangle)$ be as given in Framework \ref{frame:EPEM}. 
We start with the following lemma proved by Naotaka Kajino \cite{Kaj+}. 
\begin{lem}\label{lem:Floc.another}
	Assume that all metric balls in $(X,d)$ are relatively compact. 
	If $U$ is a non-empty open subset of $X$, then 
	\begin{align}\label{eq:Floc.another}
		&\biggl\{ u \in L^{\infty}_{\mathrm{loc}}(U,m|_{U}) \biggm|
        	\begin{minipage}{220pt}
            	$u = u^{\#}$ $m$-a.e.\ on $A$ for some $u^{\#} \in \mathcal{F}_{p}$ for each relatively compact open subset $A$ of $U$
        	\end{minipage}
        \biggr\} \nonumber \\  
        &= \biggl\{ u \in L^{\infty}_{\mathrm{loc}}(U,m|_{U}) \biggm|
        	\begin{minipage}{255pt}
            	for any $x \in U$ there exist an open neighborhood $B$ of $x$ and $u^{\#} \in \mathcal{F}_{p}$ such that $u = u^{\#}$ $m$-a.e.\ on $B$
        	\end{minipage}
        \biggr\}.  
	\end{align}  
\end{lem}
\begin{proof}
	It is obvious that the set in the left-hand side of \eqref{eq:Floc.another} is contained in the set in the right-hand side of \eqref{eq:Floc.another}. 
	To show the reverse inclusion, let 
	\[
	u \in \biggl\{ u \in L^{\infty}_{\mathrm{loc}}(U,m|_{U}) \biggm|
        	\begin{minipage}{260pt}
            	for any $x \in U$ there exist an open neighborhood $B$ of $x$ and $u^{\#} \in \mathcal{F}_{p}$ such that $u = u^{\#}$ $m$-a.e.\ on $B$
        	\end{minipage}
        \biggr\}
	\]
	and $A$ be a relatively compact open subset of $X$. 
	Since $\closure{A}^{X}$ is compact, there exist $N \in \mathbb{N}$, $(x_k,r_k) \in X \times (0,\infty)$ and $u_{k}^{\#} \in \mathcal{F}_{p}$ for each $k \in \{ 1,\dots,N \}$ such that $B(x_k,2r_k) \subset U$, $\closure{A}^{X} \subset \bigcup_{k=1}^{N}B(x_k,r_k)$ and $u = u_{k}^{\#}$ $m$-a.e.\ on $B(x_k,2r_k)$. 
	By Framework \ref{frame:EPEM}-\ref{it:Lip},\ref{it:regular}, for each $k \in \{ 1,\dots,N \}$ there exists $\varphi_{k} \in \mathcal{F}_{p} \cap \contfunc_{c}(X)$ such that $0 \le \varphi_{k} \le 1$, $\varphi_{k}|_{B(x_k,r_k)} = 1$ and $\supp_{X}[\varphi_{k}] \subset B(x_k,2r_k)$. 
	In addition, there exists $\eta \in \mathcal{F}_{p} \cap \contfunc_{c}(X)$ such that $0 \le \eta \le 1$, $\eta = 1$ on $\bigl\{ \sum_{k=1}^{N}\varphi_{k} \ge 1 \bigr\}$ and $\supp_{X}[\eta] \subset \bigl\{ \sum_{k=1}^{N}\varphi_{k} > \frac{1}{2} \bigr\}$. 
	Now we define 
	\[
	\psi_k \coloneqq \frac{\eta\varphi_{k}}{\sum_{j=1}^{N}\varphi_{j}} \in \contfunc(X), 
	\]
	where $0/0$ is considered as $0$.  
	Then $\psi_k = \eta\varphi_{k}\Upsilon\bigl(\sum_{j=1}^{N}\varphi_{j}\bigr) \in \mathcal{F}_{p}$ by Proposition \ref{prop:Lipc.list}-\ref{it:Lip-bdd}, where $\Upsilon \in \contfunc(\mathbb{R})$ is a $4$-Lipschitz map defined by 
	\[
	\Upsilon(x) \coloneqq 
	\begin{cases}
		-4x+4 \quad &\text{if $x \le \frac{1}{2}$,} \\
		\frac{1}{x} \quad &\text{if $x > \frac{1}{2}$.}
	\end{cases}
	\]
	Noting that $\sum_{k=1}^{N}\varphi_{k} \ge \indicator{A}$, we have 
	\[
	u = u\sum_{k = 1}^{N}\psi_{k} = \sum_{k = 1}^{N}u_{k}^{\#}\psi_{k} \quad \text{$m$-a.e.\ on $A$.}
	\]
	Since $u_{k}^{\#}\psi_{k} \in \mathcal{F}_{p}$ by $u \in L^{\infty}_{\mathrm{loc}}(U,m|_{U})$ and Proposition \ref{prop:Lipc.list}-\ref{it:Lip-bdd}, we see that 
	\[
	u \in \biggl\{ u \in L^{\infty}_{\mathrm{loc}}(U,m|_{U}) \biggm|
        	\begin{minipage}{220pt}
            	$u = u^{\#}$ $m$-a.e.\ on $A$ for some $u^{\#} \in \mathcal{F}_{p}$ for each relatively compact open subset $A$ of $U$
        	\end{minipage}
        \biggr\}, 
	\]
	completing the proof. 
\end{proof}

Now we show the main result in this section, which improves Proposition \ref{prop:PI-up}. 
\begin{prop}\label{prop:PI-conn}
	Assume that all metric balls in $(X,d)$ are relatively compact and that \ref{e:PI} holds. 
	Then $X$ is connected.  
\end{prop}
\begin{proof}
	Suppose that $U$ and $U'$ are disjoint non-empty open subsets of $X$ such that $X = U \cup U'$. 
	Let us fix $(x,r) \in X \times (0,\infty)$ satisfying $B(x,r) \cap U \neq \emptyset$ and $B(x,r) \cap U' \neq \emptyset$. 
	Set $u \coloneqq \indicator{U \cap B(x,A_{\mathrm{P}}r)} \in L^{p}(X,m)$. 
	Using Lemma \ref{lem:Floc.another}, Framework \ref{frame:EPEM}-\ref{it:Lip},\ref{it:regular} and the assumption that both $U$ and $U'$ are open, we can easily see that $u \in \mathcal{F}_{p,\mathrm{loc}}(B(x,A_{\mathrm{P}}r))$. 
	This yields a contradiction since \ref{e:PI} and Framework \ref{frame:EPEM}-\ref{it:EMslocal} imply that 
	\begin{align*}
		0 < \int_{B(x,r)}\abs{u - u_{B(x,r)}}^{p}\,dm 
		&\le C_{\mathrm{P}}\Psi(r)\int_{B(x,A_{\mathrm{P}}r)}d\Gamma_{p}\langle u \rangle \\
		&= C_{\mathrm{P}}\Psi(r)\bigl[\Gamma_{p}\langle u \rangle(B(x,A_{\mathrm{P}}r) \cap U) + \Gamma_{p}\langle u \rangle(B(x,A_{\mathrm{P}}r) \cap U')\bigr] \\
		&= 0. 
	\end{align*} 
	Hence $X$ has to be connected. 
\end{proof}

\noindent \textbf{Acknowledgements.} 
The estimate \eqref{e:KS} was inspired by a question to the author's talk at Applied Mathematical Analysis Seminar of Tohoku University. 
The author is grateful to organizers and members of this seminar for the warm hospitality.
He also would like to thank Takashi Kumagai and Nageswari Shanmugalingam for some illuminating conversations related to this work, which encourage him to prove Corollary \ref{cor:main1.3}. 
He also thanks Riku Anttila, Sylvester Eriksson-Bique, Naotaka Kajino and anonymous referees for some illuminating conversations, helpful comments and suggestions. The current formulation of Framework \ref{frame:EPEM}-\ref{it:EMslocal} resulted from a discussion with Riku Anttila and Sylvester Eriksson-Bique, Lemma \ref{lem:Floc.another} and its proof were communicated to the author by Naotaka Kajino in \cite{Kaj+}, and Propositions \ref{prop:PI-up} and \ref{prop:PI-conn} were inspired by a comment of a referee. 
The author was supported in part by JSPS KAKENHI Grant Number JP23KJ2011.

\noindent \textbf{Ryosuke Shimizu} \\
Waseda Research Institute for Science and Engineering, Waseda University, 3-4-1 Okubo, Shinjuku-ku, Tokyo 169-8555, Japan. \\
Graduate School of Informatics, Kyoto University, Yoshida-honmachi, Sakyo-ku, Kyoto 606-8501, Japan (current address). \\
E-mail: \texttt{r.shimizu@acs.i.kyoto-u.ac.jp}

\end{document}